\documentclass[11pt]{amsart}
\usepackage{amsfonts}
\usepackage{amsmath, amsthm, amssymb,color,bbm,array}
\usepackage{latexsym}
\usepackage[noend]{algpseudocode}
\usepackage{graphics}
\usepackage{graphicx}
\usepackage{epsfig}
\usepackage{mathrsfs}
\usepackage{subcaption}
\usepackage{color}

\numberwithin{equation}{section}
\allowdisplaybreaks

\definecolor{darkblue}{rgb}{.1, 0.1,.8}
\definecolor{darkgreen}{rgb}{0,0.8,0.2}
\definecolor{darkred}{rgb}{.8, .1,.1}
\newcommand{\E}{{\mathbb E}}

\newcommand{\bfi}{\begin{fig}}
\newcommand{\efi}{\end{fig}}

\newtheorem{lemma}{Lemma}[section]

\newtheorem{theorem}[lemma]{Theorem}

\newtheorem{proposition}[lemma]{Proposition}
\newtheorem{definition}[lemma]{Definition}
\newtheorem{corollary}[lemma]{Corollary}
\newtheorem{example}[lemma]{Example}
\newtheorem{exercise}[lemma]{Exercise}
\newtheorem{remark}[lemma]{Remark}
\newtheorem{fig}[lemma]{Figure}
\newtheorem{tab}[lemma]{Table}

\newcommand{\bth}{\begin{theorem}}
\newcommand{\ethe}{\end{theorem}}

\newcommand{\bre}{\begin{remark}\em }
\newcommand{\ere}{\end{remark}}

\newcommand{\ble}{\begin{lemma}}
\newcommand{\ele}{\end{lemma}}

\newcommand{\bde}{\begin{definition}}
\newcommand{\ede}{\end{definition}}
\newcommand{\bco}{\begin{corollary}}
\newcommand{\eco}{\end{corollary}}
\newcommand{\bpr}{\begin{proposition}}
\newcommand{\epr}{\end{proposition}}
\newcommand{\bpf}{\begin{proof}}
\newcommand{\epf}{\end{proof}}
\newcommand{\bexer}{\begin{exercise}}
\newcommand{\eexer}{\end{exercise}}
\newcommand{\bexam}{\begin{example}\rm }
\newcommand{\eexam}{\end{example}}
\newcommand{\btab}{\begin{tab}}
\newcommand{\etab}{\end{tab}}

\newcommand{\beao}{\begin{eqnarray*}}
\newcommand{\eeao}{\end{eqnarray*}\noindent}
\newcommand{\beam}{\begin{eqnarray}}
\newcommand{\eeam}{\end{eqnarray}\noindent}
\newcommand{\beqq}{\begin{equation}}
\newcommand{\eeqq}{\end{equation}\noindent}
\newcommand{\beqo}{\begin{equation*}}
\newcommand{\eeqo}{\end{equation*}\noindent}
\newcommand{\bce}{\begin{center}}
\newcommand{\ece}{\end{center}}
\newcommand{\barr}{\begin{eqnarray}}
\newcommand{\earr}{\end{eqnarray}}

\newcommand{\vague}{\stackrel{\lower0.2ex\hbox{$\scriptscriptstyle
                    \it{v} $}}{\rightarrow}}
\newcommand{\weak}{\stackrel{\lower0.2ex\hbox{$\scriptscriptstyle
                    \it{w} $}}{\rightarrow}}
\newcommand{\what}{\stackrel{\lower0.2ex\hbox{$\scriptscriptstyle
                    \it{\hat{w}} $}}{\rightarrow}}
\newcommand{\bdis}{\begin{displaymath}}
\newcommand{\edis}{\end{displaymath}\noindent}

\newcommand{\bealm}{\begin{align}}
\newcommand{\eealm}{\end{align}\noindent}
\newcommand{\tonnote}[1]{\textcolor{black}{#1}}

\newcommand{\zhipengnote}[1]{\textcolor{black}{#1}}
\newcommand{\thomasnote}[1]{\textcolor{black}{#1}}

\newcommand{\ignore}[1]{}

\begin{document}
\def\Send{S_{\text{end}}}

\title[\zhipengnote{On Optimal Exact Simulation of Max-Stable and Related Random Fields on a Compact Set}]{%
\zhipengnote{On Optimal Exact Simulation of Max-Stable and Related Random Fields on a Compact Set}}
\address{Industrial Engineering and Operations Research Department, School
of Engineering and Applied Sciences, Columbia University, 500 West 120th
Street, New York, NY 10027, U.S.A.}
\author[Z. Liu]{Zhipeng Liu}
\email{zl2337@columbia.edu}
\author[J. Blanchet]{Jose Blanchet}
\email{jose.blanchet@columbia.edu}
\author{A.B.~Dieker}
\email{ton.dieker@ieor.columbia.edu}
\author[T. Mikosch]{Thomas Mikosch}
\address{Department of Mathematics, University of Copenhagen,
Universitetsparken~5, DK-2100 Copenhagen, Denmark}
\email{mikosch@math.ku.dk}
\thanks{Support from NSF grant DMS-132055 and NSF grant CMMI-1538217 is gratefully acknowledged by J.
Blanchet. Thomas Mikosch's research is partly supported by the Danish
Research Council Grant DFF-4002-00435 \textquotedblleft Large random
matrices with heavy tails and dependence\textquotedblright . A.B.~Dieker
gratefully acknowledges support from NSF grant CMMI-1252878.}

\begin{abstract}
We consider the random field
\begin{equation*}
M(t)=\sup_{n\geq 1}\big\{-\log A_{n}+X_{n}(t)\big\}\,,\qquad t\in T\,,
\end{equation*}%
for a set $T\subset \mathbb{R}^{m}$, where $(X_{n})$ is an iid
sequence of centered Gaussian random fields on $T$ and $0<A_{1}<A_{2}<\cdots
$ are the arrivals of a general renewal process on $(0,\infty )$,
independent of $(X_{n})$. In particular, a large class of max-stable random
fields with Gumbel marginals have such a representation. Assume that one
needs $c\left( d\right) =c(\{t_{1},\ldots,t_{d}\})$ function evaluations to
sample $X_{n}$ at $d$ locations $t_{1},\ldots ,t_{d}\in T$. We provide an
algorithm which, for any $\epsilon >0$, samples $M(t_{1}),\ldots ,M(t_{d})$
with complexity $o(c(d)\,d^{\epsilon })$ \tonnote{as measured in the $L_p$ norm sense for any $p\ge 1$}. 
Moreover, if $X_{n}$ has an a.s.
converging series representation, then $M$ can be a.s. approximated with
error $\delta $ uniformly over $T$ and with complexity $O(1/(\delta \log
(1/\delta ))^{1/\alpha })$, where $\alpha $ relates to the H\"{o}lder
continuity exponent of the process $X_{n}$ (so, if $X_{n}$ is Brownian
motion, $\alpha =1/2$).
\end{abstract}

\maketitle


\today

\section{Introduction}\label{sec:intro}\setcounter{equation}{0}

Let $X$ be a centered Gaussian random field on a set $T\subseteq \mathbb{R}%
^{m}$, $m\geq 1$ and consider a sequence $\left( X_{n}\right) $ of
independent and identically distributed \ copies of $X$. In addition, let $%
\left( A_{n}\right) $ be a renewal sequence independent of $\left(
X_{n}\right) $. Under mild regularity conditions on the $X$, we will provide
an efficient Monte-Carlo algorithm for sampling the field
\begin{equation}
M(t)=\sup_{n\geq 1}\big\{-\log A_{n}+X_{n}(t)+\mu (t)\big\},\qquad t\in T\,,
\label{eq_defM}
\end{equation}%
where $\mu :T\longrightarrow \mathbb{R}$ is a bounded function.

We will design and analyze an algorithm for the exact simulation of
\begin{equation*}
M(t_{1}),\ldots ,M(t_{d})\quad
\mbox{for any choice of distinct locations
$t_{1},...,t_{d}\in T$,}
\end{equation*}%
and we will show that, in some sense, this algorithm is asymptotically
optimal as $d\rightarrow \infty $.

The algorithm proposed here shaves off a factor of order (nearly) $d$ from
the running time of any of the existing exact sampling procedures. In
particular, we will show \tonnote{that, under mild boundedness assumptions on $X$, it}
is as hard to sample $(M(t_{i}))_{i=1,\ldots ,d}$ as it is to sample $(X(t_{i}))_{i=1,\ldots ,d}$. 
Therefore, at least from a simulation point of
view, it is not more difficult to work with $M$ than with $X$. More
precisely, if it takes $O(c(d))$ units of computing time to sample $X$ at $d$
distinct locations $t_{1},\ldots ,t_{d}\in T$, then, for any given $\epsilon
>0$, it takes $o(c(d)\,d^{\epsilon })$ units to sample $M$ at the same
locations; see Theorem \ref{Thm_Main} for a precise formulation.


We illustrate this result by considering fractional Brownian motion\ $X$ on $%
T=[0,1]$. Using the circulant-embedding method (see \cite{AGbook}, Section
XI.3), we have $c(d)=O(d\log d)$ provided we sample at the dyadic points $%
t_{i}=i/2^{-m}$ for $i=1, \ldots ,2^{m}=d$ (which we call dyadic points at
level $d$). In the case of Brownian motion, one even has $c(d)=O(d)$,
corresponding to the simulation of $d$ independent Gaussian random
variables. Thus, in the case of fractional Brownian motion\ on $[0,1]$ we
provide an algorithm for sampling $M$ at the dyadic points at level $d$ in
[0,1] with complexity $o(d^{1+\epsilon })$ for any $\epsilon >0$; see \cite[%
Sec.~XI.6]{AGbook}.

Moreover, if $X$ has a series representation a.s.~converging uniformly on $T$
(such as the L\'{e}vy-Ciesielski representation for Brownian motion, see
\cite[Sec.~3.1]{partzsch:schilling:2012}\thomasnote{)}, we also propose an approximate
simulation procedure for $M$ with a user-defined (deterministic) bound on
the error which holds with probability one uniformly throughout $T$. More
precisely, for any $\delta >0$, the procedure that we present outputs an
approximation $M_{\delta }$ to $M$ such that
\begin{equation}
\sup_{t\in T}|M(t)-M_{\delta }(t)|\leq \delta \qquad \mathrm{a.s.}
\label{eq_BND}
\end{equation}%
The results concerning (\ref{eq_BND}) are reported in Theorem \ref{Thm_TES}.
The method of designing a family $(M_{\delta })_{\delta >0}$ such that (\ref%
{eq_BND}) holds is known as \emph{Tolerance Enforced Simulation} (TES) or
\emph{$\delta $-strong simulation}; see \cite{BCD} and \cite{PJR} for
details. Note that a TES algorithm enforces a strong (almost sure) guarantee
without knowledge of any specific set of sampling locations. This is a
feature which distinguishes TES from more traditional algorithms in the
broad literature on simulation of random fields and processes.

As will be explained later, the evaluation of $M_{\delta }(t)$ for fixed $t$
takes $O(1)$ units of computing time while the construction of the process $%
M_{\delta }$ will often take $O(1/(\delta \log (1/\delta ))^{2})$ units. The
latter result holds under assumptions on the convergence\ of the series
representation of $X$ which, in particular, are satisfied for Brownian
motion $X$. In the latter case, the proposed procedure achieves a complexity
of order $O(d)$ for the exact sampling of $M$ on the dyadic points at level $%
d$ (because the series truncated at level $d$ is exact on the dyadic points
at level $d$). Therefore, the exact sampling procedure based on Theorem \ref%
{Thm_TES} applied to the dyadic points at level $d$ is optimal because it
takes $O(d)$ computational cost to sample $X$ at $d$ dyadic points.
Moreover, the convergence rate \ignore{in (\ref{eq_BND})}\zhipengnote{of the TES algorithm} is also optimal in the
Brownian case. In order to obtain a uniform error of order $O(\delta )$, one
requires to discretize Brownian motion using a grid of size $O(1/(\delta
\log (1/\delta ))^{2})$; see \cite{AGP}.

Our results are mainly motivated by application to the simulation of
max-stable random fields. Indeed, if $(A_{i})$ is the arrival sequence of a
unit rate Poisson process on $(0,\infty )$, $M$ is a \emph{max-stable process%
} in the sense of de Haan \cite{Haan}. This means, in particular, that the
distribution of $M(t)$ for any fixed $t\in T$ has a Gumbel distribution
which is one of the max-stable distributions. The latter class of
distributions consists of the non-degenerate limit distributions for the
suitably centered and scaled partial maxima of an iid sequence; see for
example \cite{EKMBK}. The non-Gumbel max-stable processes with Fr\'{e}chet
or Weibull marginals are obtained from the representation (\ref{eq_defM}) by
suitable monotone transformations. We also mention that de Haan \cite{Haan}
already proved that max-stable processes with Gumbel marginals have
representation (\ref{eq_defM}), where $X$ may have a rather general
dependence structure not restricted to Gaussian $X$. However, the case of
Gaussian $X$ has attracted major attention. The case of Brownian $X$ was
treated in \cite{BR}\thomasnote{; it} is known as the 
\emph{Brown-Resnick process}. 
\thomasnote{In the paper \cite{KZH}, the case of a general Gaussian process $X$ with stationary increments was treated, including}
\ignore{Other
examples includes }the case of a Gaussian process $X$ defined on a
multidimensional set $T$ often referred to as \emph{Smith model}.
\thomasnote{It is}
used in environmental applications for modeling storm profiles; \thomasnote{see for example} \cite{Smith}. General characterizations, including spectral
representations and further properties, have been obtained as well; for
example, see \cite{Sch} and \cite{KZH}. However, the explicit joint distribution
of the max-stable process is in general not tractable. Because max-stable
processes are generated as weak limits of maxima of iid random fields,
max-stable models are particularly suited for modeling extremal events in
\thomasnote{spatio-temporal} contexts. These include a wide range of applications of
environmental type, for example, extreme rainfall \cite{HZ} and extreme
\thomasnote{temperature~\cite{Tetal}.}

Recently, several exact sampling procedures for $M$ have been proposed and
studied in the literature. In \cite{DM}, an elegant and easy-to-implement
procedure was proposed for the case in which $X$ has stationary increments.
Such a procedure has a computational complexity at least of order $O(c(d)\,d)$; see 
Proposition~4 in \cite{DEO}. So, for example, if $X$ is fractional
Brownian motion, the procedure takes at least $O(d^{2}\log d)$ units of
computing time to produce $d$ dyadic points of $M$ in $[0,1]$.

Another exact simulation method for $M$ was recently proposed in \cite{DEO}.
It also has complexity $O(c(d)\,d)$ (see Proposition 4 in \cite{DEO}), thus
the procedure in \cite{DEO} takes $O(d^{2}$) for fractional Brownian motion
(neglecting the contribution of logarithmic factors). This method is based
on the idea of simulating the extremal functions. It is completely different
from the approach taken here. Additional work concentrates on max-stable
processes which satisfy special characteristics. \zhipengnote{For example, \cite{Sch}
proposed an exact simulation algorithm for the moving maxima model under
suitable uniformity conditions. }

\tonnote{Another recent development is \cite{OSZ}, where the authors discuss an
exact sampling algorithm for max-stable fields using the so-called normalized spectral representation. 
If the normalized spectral functions can be sampled with cost $c_{\text{\tiny{NS}}}(d)$, 
then the algorithm in \cite{OSZ} samples the max-stable field exactly with complexity $O(c_{\text{\tiny{NS}}}(d))$. 
However, for Gaussian-based max-stable fields, it is an open problem to devise exact sampling algorithms for the 
normalized spectral function, and it is unclear how $c_{\text{\tiny{NS}}}(d)$ compares with the complexity $c(d)$ of sampling $X$.}

An important difference between our method and those in \cite{DM} and \cite%
{DEO} is the following: Both \cite{DM} and \cite{DEO} take advantage of
representations or structures which allow to truncate the infinite
max-convolution in (\ref{eq_defM}) while preserving the simple Gaussian
structure of the number of terms in the truncation. Because the simple
structure of these terms is preserved, the number of terms in the truncation
increases at least linearly in $d$. In contrast, we are able to truncate the
number of terms in the infinite max-convolution uniformly in $d$. While the
terms in the truncation have a slightly more complex structure (they are no
longer iid Gaussian), they are still quite tractable from a simulation
standpoint.

This paper is organized as follows: In Section~\ref{SEC_MAIN_THM} we present
our main result and in Section \ref{sec:strategy} we\ discuss our general
strategy, based on \emph{milestone events} or \emph{record-breakers.} The
record-breaking strategy is illustrated in Section~\ref%
{Section_Milestone_RW_new} in the setting of random walks, which is needed in
our context due to the presence of $(A_{n})$ in $M$. Then we apply the
record-breaking strategy to the setting of maxima of Gaussian random vectors
with focus on Section \ref{SECTION_MILESTONE}: This section describes the
main algorithmic developments of the paper. A complexity analysis is
performed in Section~\ref{sec:complexity}. We introduce and analyze a TES
algorithm in Section~\ref{SECT_TES}. Finally, in Section~\ref{SEC_NUMERICS},
we conclude our paper with a series of empirical comparison results.

\section{Main Result\label{SEC_MAIN_THM}}\setcounter{equation}{0}



This section provides a formal statement of the main result and its
underlying assumptions. We assume that $(A_{n})_{n\geq 0}$ is a renewal
sequence, as mentioned in the Introduction. In particular, $A_{0}=0$, and $%
A_{n}=\tau _{1}+\cdots +\tau _{n}$, $n\geq 1$, where $\left( \tau
_{i}\right) $ is an iid sequence of positive random variables, independent
of $(X_{n})$.

\bigskip

We introduce the following technical assumptions applicable to $\left(
A_{n}\right) $:

\begin{enumerate}
\item[A1)] For any $\gamma <\mathbb{E}\tau _{1}$, there exists some $\theta
_{\gamma }>0$ such that $\mathbb{E}[\exp (\theta _{\gamma }(\gamma -\tau
_{1}))]=1$.

\item[A2)] It is possible to sample step sizes under the nominal probability
measure as well as under the exponentially tilted distribution
\begin{equation*}
\mathbb{E}\big[\exp (\theta _{\gamma }\,(\gamma -\tau _{1}))\mathbf{1}(\tau
_{1}\in dt)\big].
\end{equation*}
\end{enumerate}

\bigskip

We also introduce the following assumptions \ignore{applicable to}\thomasnote{on} the Gaussian field
$\left( X( t) \right) _{t\in T}$. \ignore{and the function $\mu $.}

\begin{enumerate}
\item[B1)] $\mathbb{E}[X(t)]=0$.

\item[B2)] $\mathbb{E}[\exp \left( p\sup_{t\in T}X(t) \right)] <\infty $ for any $p \ge 1$.
\end{enumerate}

\bigskip
\begin{remark}\rm
By Borell's inequality \cite[Thm.~2.1.1]{AT}, \zhipengnote{if $T$ is bounded}, a sufficient condition for B2)
is
\begin{equation*}
\mathrm{Var}(X(s)-X(t))\le c |s-t|^\beta
\end{equation*}
for any $s,t \in T$ and some $c>0$, $\beta >0$. \thomasnote{Define $\sigma ^{2}(t)=%
\mathrm{Var}\left( X( t) \right)$. Then, under B1) and
B2),}
\begin{equation*}
\sup_{t\in T}\sigma ^{2}(t)
=\sup_{t\in T}\mathbb{E}[X( t)^{2}]\leq \mathbb{E}\left[\sup_{t\in
T}X( t) ^{2}\right]<\infty .
\end{equation*}
\end{remark}
We also assume that sampling $(X(t_{i}))_{i=1,\ldots ,d}$ costs $%
c(\{t_{1},\ldots ,t_{d}\})\geq d$ units of operations. In this paper, a
single operation can be any single arithmetic operation, generating a
uniform random variable, calculating a Gaussian cumulative probability
function, comparing any two numbers, or retrieving a Gaussian quantile
value. For simplicity in the notation, we shall simply write $c\left(
d\right) =c(\{t_{1},\ldots ,t_{d}\})$. The locations $t_{1},...,t_{d}$ will
be assumed given throughout our development.

\bigskip

\noindent The following is our performance guarantee for our final
algorithm, \textbf{Algorithm M}, presented in Section~\ref{sec:complexity}.
A crucial part of the theorem is that the points $t_1,\ldots,t_d$ for any $%
d\ge 1$ lie in a fixed \thomasnote{set~$T$.}

\begin{theorem}
\label{Thm_Main} \thomasnote{Assume the conditions {\rm A1), A2),} and 
{\rm B1), B2).} Then \textbf{Algorithm M} 
outputs $M(t_{1}),\ldots ,M(t_{d})$ without any bias, and 
the total number $R$} of operations in the execution of \thomasnote{this algorithm}
\ignore{\textbf{Algorithm M%
}}satisfies 
$\mathbb{E}[R^{p}]=o(d^{\epsilon }c(d)^{p})$ for any $p\geq 1$
and $\epsilon >0$.


\end{theorem}

\section{Building Blocks For Our Algorithm}\setcounter{equation}{0}

\label{sec:strategy} This section serves as a roadmap for the algorithmic
elements behind our approach. We start with a few definitions:

\begin{equation*}
\overline{X}_{n}=\max_{i=1,\ldots ,d}X_{n}(t_{i}),\qquad \underline{X}%
_{n}=\min_{i=1,\ldots ,d}X_{n}(t_{i}).
\end{equation*}%
We shall use $\overline{X}$ and $\underline{X}$ to denote generic copies of $%
\overline{X}_{n}$ and $\underline{X}_{n}$, respectively. We also set $%
\overline{\mu }=\max_{i=1,\ldots ,n}\mu (t_{i})$ and $\underline{\mu }%
=\min_{i=1,\ldots ,n}\mu (t_{i})$.%

Our algorithm relies on three random times \thomasnote{which are finite a.s. 
They depend on parameters $a\in (0,1]$, $C\in\mathbb R$, $0<\gamma<\E[A_1]$ to be chosen later.}
\begin{enumerate}
\item
\thomasnote{$N_{X}=N_{X}(a,C)$:}
for all $n>N_{X}$,
\begin{equation*}
\overline{X}_{n}\leq a\log n+C.
\end{equation*}
\thomasnote{A straightforward Borel-Cantelli argument shows that $N_X$ is
finite.}
\item
\thomasnote{$N_{A}=N_{A}(\gamma )$: }
for all $n>N_{A}$,
\begin{equation}  \label{eq:boundAgamma}
A_{n}\geq \gamma n.
\end{equation}
\item
\thomasnote{$N_{a }=N_{a}(\gamma ,a,C)$: }
for all $n>N_a$,
\begin{equation}  \label{eq:defNmu}
n\gamma \geq A_{1}\,n^{a}\exp (C-\underline{X}_{1}).
\end{equation}
\end{enumerate}
\thomasnote{Applying the defining properties of these random times,} 
we find that for $%
n>N:=\max (N_{A},N_{X},N_a)$ and any \thomasnote{$t\in\{t_{1},\ldots ,t_{d}\}$},
\begin{eqnarray*}
-\log A_{n}+X_{n}(t) &\leq &-\log A_{n}+\overline{X}_{n} \\
&\leq &-\log A_{n}+a\log n+C \\
&\leq &-\log (n\gamma )+a\log n+C \\
&\leq &-\log A_{1}+\underline{X}_{1} \\
&\leq &-\log A_{1}+X_{1}(t).
\end{eqnarray*}

We conclude that, for $t\in \{t_1,\ldots,t_d\}$,
\begin{eqnarray}  \label{eq:truncationsup}
\sup_{n\ge 1} \left\{-\log A_n + X_n(t) + \mu(t) \right\} = \max_{1\le n\le
N} \left\{-\log A_n + X_n(t) + \mu(t) \right\},
\end{eqnarray}
and thus we can sample $M(t_1),\ldots,M(t_d)$ with computational complexity 
\thomasnote{$N c(d)$} plus the overhead to identify $N_A$ and $N_X$.

From an algorithmic point of view, the key is the simulation of the random
variables $N_X$, $N_A$, and $N_a$. \thomasnote{If we know how to
simulate these quantities, relation \eqref{eq:truncationsup} indicates
that we must be able to simulate the sequences $(A_n)$ and $(X_n)$
up to and jointly with $N$ which heavily depends on both sequences.}


\begin{remark}\label{rem:remove_assump}\rm
\thomasnote{Assumptions A1) and A2) can be removed without loss of generality. To see this, we first observe that for any $r>0$,}
\ignore{First, observe that we have for any $r>0$ we have that }
$\tau_{i}(r) =\min \left(
\tau _{i},r\right) \leq \tau _{i}$ and, therefore,
\begin{equation*}
A_{n}( r ) =\tau _{1}( r) +\cdots+\tau _{n}( r ) \leq A_{n}.
\end{equation*}%
Moreover, we can select $r>0$ so that $\gamma <\mathbb{E}\left[ \tau _{i}( r
) \right] <\mathbb{E}[ \tau _{i}] $. \thomasnote{Hence} \ignore{which means that}we can use $\left(
A_{n}( r) \right) _{n\geq 1}$ \ignore{in order}to find $N_{A}$ satisfying
\begin{equation*}
A_{n}>A_{n}( r) >\gamma\,n .
\end{equation*}%
\ignore{which implies (\ref{eq:boundAgamma}).} Because $0\leq \tau _{n}( r) \leq r$,
the moment generating function of $\tau _{n}( r) $ exists on the whole real
line. By convexity, one can always choose $\theta _{\gamma }$ which
satisfies $\mathbb{E}[\exp (\theta _{\gamma } (\gamma-\tau_1(r)))]=1$, as
long as $\mathrm{Var}\left( \tau _{i}( r) \right) >0$ (i.e. if $\tau _{i}>0$
is non-deterministic, by choosing $r>0$ large enough). If $\tau _{i}$ is
deterministic, the strategy can be implemented directly, that is, we can
simply select $N_{A}$ deterministic.
\zhipengnote{Once we find $N_A$, we can recover $(A_n)_{n \le N_A}$ from $(A_n(r))_{n \le N_A}$ by replacing $\tau_n(r)$ with an independent sample of $\tau_n$ given $\tau_n\ge r$, for any $n \le N_A$ such that $\tau_n(r)=r$, and keeping $\tau_n(r)$ if it is less than $r$.}

Given our previous discussion, we might concentrate on how to sample from an
exponentially tilted distribution of a random variable with compact support,
which may require evaluating the moment generating function in closed form.
Sampling from an exponentially tilted distribution is straightforward for
random variables with finite support. So, the strategy can be implemented
for $\left\lfloor \tau _{i}\left( r\right) \Delta \right\rfloor /\Delta <\tau _{i}\left(
r\right) $, \zhipengnote{where $\left\lfloor \cdot \right\rfloor$ is 
the round-down operator}, picking $\Delta >0$ sufficiently small so that $\mathbb{E}\left[ %
\left\lfloor \tau _{i}\left( r\right) \Delta \right\rfloor/\Delta \right] >\gamma $.
Once $\left\lfloor \tau _{i}\left( r\right) \Delta \right\rfloor$ is sampled we can
easily simulate $\tau _{i}\left( r\right) $ using acceptance/rejection. 
The details of this idea are explained in \cite{BW}.
\end{remark}

\section{Sampling a Random Walk up to a Last Passage Time\setcounter{equation}{0}}\setcounter{equation}{0}

\label{Section_Milestone_RW_new}

\zhipengnote{In this section, we discuss the simulation of the random time $N_A$ 
jointly with the sequence $(A_n)_{n\ge 0}$.}\ignore{Not true: Since $A_n$ is a random walk with exponentially distributed increments,} 
\thomasnote{We lead this discussion in the context of a general 
random walk} $(S_{n})_{n\geq 0}$ starting from the origin with
negative drift. \thomasnote{It is} eventually negative almost surely. 
\thomasnote{We review} an algorithm from \cite{BS} for finding
a random time $N_{S}$ such that $S_{n}<0$ for all $n>N_{S}$.  Our aim is to
develop a sampling algorithm for $(S_{1},\ldots ,S_{N_{S}+\ell })$ for any
fixed $\ell \geq 0$. Our discussion here provides a simpler version of the
algorithm in \cite{BS} and allows us to provide a self-contained development
of the whole procedure \thomasnote{for sampling} $M\left( t_{1}\right) ,\ldots,M\left(
t_{d}\right) $.

The algorithm is based on alternately sampling upcrossings and downcrossings
of the level $0$. We write $\xi _{0}^{+}=0$ and, for $i\geq 1$, we
recursively define
\begin{equation*}
\xi _{i}^{-}=%
\begin{cases}
\inf \{n\geq \xi _{i-1}^{+}:S_{n}<0\} & \text{if }\xi _{i-1}^{+}<\infty \\
\infty & \text{otherwise}%
\end{cases}%
\end{equation*}%
together with
\begin{equation*}
\xi _{i}^{+}=%
\begin{cases}
\inf \{n\geq \xi _{i}^{-}:S_{n}\geq 0\} & \text{if }\xi _{i}^{-}<\infty \\
\infty & \text{otherwise}.%
\end{cases}%
\end{equation*}%
As usual, in these definitions the infimum of an empty set should be
interpreted as $\infty $. Writing
\begin{equation*}
N_{S}=\sup \{\xi _{n}^{-}:\xi _{n}^{-}<\infty \},
\end{equation*}%
\thomasnote{and keeping in mind that $(S_n)$ starts as zero and 
has negative drift,
we have by construction  $0\leq N_{S}<\infty $ almost surely, and for  $n>N_{S}$,
$S_{n}\leq 0$.} \ignore{\thomasnote{delete}\tonnote{in view of the negative drift} \thomasnote{delete: this sentence is superfluous}}
\ignore{(it even holds for $%
n=N_{S}$ but we do not include this for expository reasons).}The random
variable $N_{S}-1$ is an upward last passage time:
\begin{equation*}
N_{S}-1=\sup \{n\geq 0:S_{n}\geq 0\}.
\end{equation*}%

We write $\mathbb{P}_{x}$ for the distribution of the random walk starting from $x\in
\mathbb{R}$, so that $\mathbb{P}=\mathbb{P}_{0}$. We assume the existence of
\emph{Cram\'{e}r's root}, $\theta >0$, satisfying \thomasnote{$\mathbb{E}[\exp (\theta
S_{1})]=1$.} Also assume that we can sample a random walk starting from $x$
under $\mathbb{P}_{x}^{\theta }$, which is defined with respect to $\mathbb{P%
}_{x}$ through an exponential change of measure: \thomasnote{on} the $\sigma $-field
generated by $S_{1},\ldots ,S_{n}$ we have
\begin{equation*}
\frac{d\mathbb{P}_{x}}{d\mathbb{P}_{x}^{\theta }}=\exp (-\theta (S_{n}-x)).
\end{equation*}%
Under $\mathbb{P}_{x}^{\theta }$, the random walk $(S_{n})$ has positive
drift.

The rest of this section is organized as follows: \thomasnote{
\begin{itemize}
\item
In Section~\ref{sec:sampledownupcrossing} we discuss 
sampling of downcrossing and upcrossing segments of the random walk. 
\item
In Section~\ref{sec:beyondNS} we explain
how to sample beyond $N_S$. 
\item
In Section~\ref{sec:fullalgorithmA} we presents our
full algorithm for sampling $(S_1,\ldots, S_{N_S+\ell})$.
\end{itemize}}

\subsection{Downcrossings and upcrossings}

\label{sec:sampledownupcrossing} To introduce the algorithm, we first need
the following definitions:
\begin{equation*}
\tau ^{-}=\inf \{n\geq 0:S_{n}<0\},\quad \quad \tau ^{+}=\inf \{n\geq
0:S_{n}\geq 0\}.
\end{equation*}

For $x\ge 0$, it is immediate that we can sample a downcrossing segment $%
S_1,\ldots,S_{\tau^-}$ under $\mathbb{P}_x$ due to the negative drift, and
we record this for later use in a pseudocode function. \emph{Throughout this
paper, `sample' in pseudocode stands for `sample independently of anything
that has been sampled already.'}

\textbf{Function }\textsc{SampleDowncrossing}\textbf{(}$x$\textbf{): Samples
}$(S_{1},\ldots ,S_{\tau ^{-}})$\textbf{\ under }$\mathbb{P}_{x}$\textbf{\
for }$x\geq 0$

Step 1: Return sample $(S_{1},\ldots ,S_{\tau ^{-}})$ under $\mathbb{P}_{x}$.

Step 2: EndFunction

\bigskip

Sampling an upcrossing segment is much more challenging because it is
possible that $\tau ^{+}=\infty $, so an algorithm needs to be able to
detect this event within a finite amount of computing resources. For this
reason, we understand sampling an upcrossing segment under $\mathbb{P}_{x}$
for $x<0$ to mean that an algorithm outputs $(S_{1},\ldots ,S_{\tau ^{+}})$
if $\tau ^{+}<\infty $, and otherwise it outputs `degenerate.'

Our algorithm is based on importance sampling and exponential tilting,
techniques that are widely used for rare event simulation \cite[p.~164]%
{AGbook}. Under Assumption A1), it is well-known that $\mathbb{E}%
_{x}^{\theta }[\tau ^{+}]<\infty $; for instance, see \cite[p.~231, Cor.~4.4]{Abook}.
In particular, the expected time to simulate $(S_{1},\ldots ,S_{\tau ^{+}})$
is finite under $\mathbb{P}_{x}^{\theta }$ for any $x<0$.

The following proposition is the key to our algorithm.

\begin{proposition}
\label{Prop_BL_SI_1} Let $x<0$. Suppose there exists some $\theta>0$ with $%
\mathbb{E}[\exp(\theta S_1)]=1$. With $U$ being a standard uniform random
variable independent of $(S_n)$ under $\mathbb{P}_x^\theta$, we have the
following:

\begin{enumerate}
\item The law of $\mathbf{1}(\tau^+<\infty)$ under $\mathbb{P}_x$ equals the
law of $\mathbf{1}(U\le \exp(-\theta (S_{\tau^+}-x)))$ under $\mathbb{P}%
^\theta_x$.

\item The law of $\tau^+$ given $\tau^+<\infty$ under $\mathbb{P}_x$ equals
the law of $\tau^+$ given $U\le \exp (-\theta (S_{\tau^+} -x))$ under $%
\mathbb{P}_x^\theta$.

\item For any $k\ge 1$, the law of $(S_{1} ,...,S_k)$ given $\tau^+=k$ under
$\mathbb{P}_x$ equals the law of $(S_{1} ,\ldots ,S_k)$ given $U\leq \exp
(-\theta (S_{\tau^+}-x) ) $ and $\tau^+=k$ under $\mathbb{P}_x^\theta$.
\end{enumerate}
\end{proposition}

\begin{proof}
For any integer $k\geq 1$ and Borel sets $B_{1},B_{2},\ldots ,B_{k}$, we
have
\begin{eqnarray*}
\lefteqn{\mathbb{P}_{x}\left( S_{1}\in B_{1},\ldots ,S_{k}\in B_{k},\tau
^{+}=k\right) } \\
&=&\mathbb{E}_{x}^{\theta }\big[\exp (-\theta \,(S_{k}-x))\mathbf{1}%
(S_{1}\in B_{1},\ldots ,S_{k}\in B_{k},\tau ^{+}=k)\big] \\
&=&\mathbb{E}_{x}^{\theta }\big[\mathbf{1}(U\leq \exp (-\theta \,(S_{\tau
^{+}}-x)))\,\mathbf{1}(S_{1}\in B_{1},\ldots ,S_{k}\in B_{k},\tau ^{+}=k)%
\big].
\end{eqnarray*}%
All claims are elementary consequences of this identity, upon noting that $%
\tau^{+}<\infty $ under $\mathbb{P}_{x}^{\theta }$.
\end{proof}

This proposition immediately yields the following algorithm.

\bigskip

\textbf{Function} \textsc{SampleUpcrossing}($x$): \textbf{Samples }$%
(S_{1},\ldots ,S_{\tau ^{+}})$\textbf{\ under }$\mathbb{P}_{x}$\textbf{\ for
}$x<0$

Step 1: $S\leftarrow $ sample $(S_{1},\ldots ,S_{\tau ^{+}})$ under $\mathbb{%
P}_{x}^{\theta }$

Step 2: $U\leftarrow $ sample a standard uniform random variable

Step 3: If {$U\leq \exp (-\theta (S_{\tau ^{+}}-x))$}

Step 4:\qquad Return $S$

Step 5: Else

Step 6:\qquad Return `degenerate'

Step 7: EndIf

Step 8: EndFunction


\subsection{Beyond $N_S$}

\label{sec:beyondNS} We next describe how to sample $(S_1,\ldots, S_{\tonnote{\ell}})$ from
$\mathbb{P}_{x}$ conditionally on $\tau ^{+}=\infty $ for $x<0$. Because $\tau
^{+}=\infty $ is equivalent to $\sup_{k\leq \ell }S_{k}<0$ and $\sup_{k>\ell
}S_{k}<0$ for any $\ell \geq 1$, after sampling $S_{1},\ldots ,S_{\ell }$,
by the Markov property we can use \textsc{SampleUpcrossing}$(S_{\ell })$ to
verify whether or not $\sup_{k>\ell }S_{k}<0$. This observation immediately
yields an acceptance/rejection algorithm that achieves our goal.

\bigskip

\textbf{Function }\textsc{SampleWithoutRecordS}$\left( {x,\ell }\right) $%
\textbf{: Samples }$(S_1,\ldots,S_\ell)$\textbf{\ from }$\mathbb{P}_{x}$%
\textbf{\ given }$\tau ^{+}=\infty $\textbf{\ for }$\ell \geq 1$\textbf{, }$%
x<0$

Step 1: Repeat

Step 2:$\qquad S\leftarrow $ sample $(S_1,\ldots,S_\ell)$ under $\mathbb{P}%
_{x}$

Step 3: Until $\sup_{1\leq k\leq \ell }S_{k}<0$ and \textsc{SampleUpcrossing}%
$(S_{\ell })$ is `degenerate'

Step 4: Return $S$

Step 5: EndFunction

\subsection{Sampling a random walk until a last passage time}

\label{sec:fullalgorithmA} We summarize our findings in this section in our
full algorithm for sampling $(S_{0},\ldots ,S_{N_{S}+\ell })$ under $\mathbb{%
P}$ given some $\ell \geq 0$. The validity of the algorithm is a direct
consequence of the strong Markov property.

\subparagraph{\textbf{Algorithm S: Samples }$(S_{0},\ldots ,S_{N_{S}+\ell })$%
\textbf{\ under }$\mathbb{P}$\textbf{\ for }$\ell \geq 0$}

\label{alg:N_A}

\# We use $S_{\text{end}}$ to denote the last element of $S$.

Step 1: $S\leftarrow \lbrack 0]$

Step 2: \ Repeat

Step 3:\qquad DowncrossingSegment $\leftarrow $ \textsc{SampleDowncrossing}$%
(S_{\text{end}})$

Step 4:$\qquad S\leftarrow \lbrack S,$ Do$\text{wncrossingSegment}]$

Step 5:\qquad UpcrossingSegment $\leftarrow $\textsc{SampleUpcrossing}$(S_{%
\text{end}})$

Step 6:\qquad If{\ UpcrossingSegment is not `degenerate'}

Step 7:$\qquad \qquad S\leftarrow \lbrack S,\text{UpcrossingSegment}]$

Step 8: \qquad EndIf

Step 9: \ \ Until {UpcrossingSegment is `degenerate'}

Step 10: If {$\ell >0$}

Step 11: $\qquad S\leftarrow \lbrack S,$ \textsc{SampleWithoutRecordS}$(S_{%
\text{end}},\ell )]$

Step 12: EndIf

\section{Record-Breaker Technique for the Maximum of a Gaussian Field}\setcounter{equation}{0}

\label{SECTION_MILESTONE}

After the excursion to random walks in Section~\ref{Section_Milestone_RW_new}
we return to the main theme of this paper. In particular, we stick to the
notation and assumptions of Section~1-3. Define \tonnote{$\eta_0=n_0$ for some fixed $n_0$ to be defined later.}
Let $(X_n)_{n\ge 1}$ be iid copies of $X$ and define, for $i\ge 1$, a
sequence of \emph{record-breaking times} $(\eta_i)$ through
\begin{equation*}
\eta_i =
\begin{cases}
\inf\{n>\eta_{i-1}: \overline X_n> a\log n +C\} & \text{if }
\eta_{i-1}<\infty \\
\infty & \text{otherwise.}%
\end{cases}%
\end{equation*}

It is the aim of this section to develop a sampling algorithm for $%
(X_1,\ldots,X_{N_X+\ell})$ for any fixed $\ell\ge 0$, where
\begin{equation*}
N_X = \max\{\eta_i: \eta_i<\infty\}.
\end{equation*}

Here and in what follows, we write $X_i$ for \tonnote{a} sample path at the given points
$t_1, \ldots, t_d \in T$. 
Section~\ref{sec:singlerecord} first discusses an algorithm to sample $(X_n)$
up to a single record. For this algorithm to work, $n_0$ needs to be large
enough so that $\mathbb{P}(\overline X>a\log n+C)$ is controlled for every $%
n>n_0$; the choice of $n_0$ is also discussed in Section~\ref%
{sec:singlerecord}. Section~\ref{sec:beyondNX} describes how to sample $%
(X_n) $ beyond the last record-breaking time. Section~\ref%
{sec:fullalgorithmX} presents our algorithm for sampling $(X_1,\ldots,
X_{N_X+\ell})$.

\subsection{Breaking a single record}

\label{sec:singlerecord} \ignore{Fix some integer $n_{0}\geq 0$, and define}\tonnote{We define for $n\ge n_0$,}
\begin{equation*}
T_{\tonnote{n}}=\inf \{k\geq 1:\overline{X}_{k}>a\log (\tonnote{n}+k)+C\}.
\end{equation*}%
We describe an algorithm that outputs `degenerate' if $T_{\tonnote{n}}=\infty $
and $(X_{1},\ldots ,X_{T_{\tonnote{n}}})$ if $T_{\tonnote{n}}<\infty $. Ultimately, the
strategy is based on \thomasnote{acceptance/rejection.} We will eventually sample $%
T_{\tonnote{n}}$ given \ignore{that}$T_{\tonnote{n}}<\infty $ using a suitable random variable $K$ \tonnote{as a proxy}\ignore{\thomasnote{too many ``proposal'', proxy?} }
with probability mass function $g_{n_{0}}$ \ignore{as a proposal}, \zhipengnote{which we discuss later in this subsection}. 
\ignore{However, in}\tonnote{In} order to apply this 
\thomasnote{acceptance/rejection} strategy, we need to introduce
auxiliary sampling distributions.

Our algorithm makes use of a measure $\mathbb{P}^{(n)}$ that is designed to
appropriately approximate the conditional distribution of $X$ given $%
\overline{X}>a\log n+C$, which is defined through
\begin{equation*}
\frac{d\mathbb{P}^{(n)}}{d\mathbb{P}}(x)=\frac{\sum_{i=1}^{d}\mathbf{1}%
(x(t_{i})>a\,\log n+C)}{\sum_{i=1}^{d}{\mathbb{P}(X(t_{i})>a\,\log n+C)}}.
\end{equation*}%
\thomasnote{For any index $j\in \{1,\ldots ,d\}$ and  $t\in
\{t_{1},\ldots ,t_{d}\}$, define
$w^{j}(t)=\mathrm{Cov}(X(t),X(t_{j}))/\mathrm{Var}(X(t_{j}))$. Since $X$ is
centered Gaussian
$X(t)-w^{j}X(t_{j})$ and $X(t_j)$ are uncorrelated, hence independent.} 
\ignore{
For any given 
index $j\in \{1,\ldots ,d\}$, the vector $X-w^{j}X(t_{j})$ 
is independent of $X(t_{j})$\ignore{ \thomasnote{delete:}\tonnote{since $X$ is Gaussian}, where $w^{j}$ is defined as $%
w^{j}(t)=\mathrm{Cov}(X(t),X(t_{j}))/\mathrm{Var}(X(t_{j}))$ for $t\in
\{t_{1},\ldots ,t_{d}\}$,} }
Now one readily verifies that the following algorithm
outputs samples from $\mathbb{P}^{(n)}$. \thomasnote{Here and in what follows,} 
$\Phi$ is the \thomasnote{standard normal distribution function.}

\bigskip

\textbf{Function } \textsc{ConditionedSampleX} $\left(a,C,n\right)$\textbf{:
Samples $X$ from} $\mathbb{P}^{(n)}$

Step 1: $\nu\gets$ sample with probability mass function
\begin{equation*}
\mathbb{P}(\nu =j)=\frac{\mathbb{P}(X(t_j)>a\log n+C)}{\sum_{i=1}^{d}{%
\mathbb{P}(X(t_i)>a\log n+C)}}
\end{equation*}

Step 2: $U\gets $ sample a standard uniform random variable

Step 3: $X(t_\nu)\gets \sigma(t_\nu)\Phi^{-1}\left(U + (1-U)\Phi\left(\frac{%
a\log n+C}{\sigma(t_\nu)}\right)\right)$ \# Conditions on $X(t_\nu)>a\log
n+C $

Step 4: $Y\gets $ sample of $X$ under $\mathbb{P}$

Step 5: Return $Y- w^\nu Y(t_\nu) + \zhipengnote{w^\nu} X(t_\nu)$

Step 6: EndFunction

\bigskip

We are now ready to see how \textsc{ConditionedSampleX} is used to sample
until the first record.

\bigskip

\textbf{Function} \textsc{SampleSingleRecord} $\left(a,C,\ignore{n_0,n_1}\tonnote{n}\right)$%
\textbf{: Samples} $(X_1,\ldots,X_{T_{\tonnote{n}}})$ \textbf{for} $a\in (0,1], C\in\mathbb{R}, 
\ignore{n_1\ge n_0\ge 0}\tonnote{n\ge n_0}$

Step 1: $K\gets$ sample from pmf $g_{n_0}$

Step 2: $(X_1,\ldots,X_{K-1})\gets $ iid sample under $\mathbb{P}$

Step 3: $X_K\gets \textproc{ConditionedSampleX}(a,C,\tonnote{n}+K)$

Step 4: $U\gets$ sample a standard uniform random variable

Step 5: If {$\overline X_k \le a \log (\tonnote{n}+k) + C$ for $k=1,\ldots,K-1$ and 
$U\, {g_{n_0}(K)} \leq {d\mathbb{P}/d\mathbb{P}^{(\tonnote{n}+K)}(X_K)}$}

Step 6: \qquad Return $(X_1,\ldots,X_K)$

Step 7: Else

Step 8: \qquad Return `degenerate'

Step 9: EndIf

Step 10: EndFunction

\bigskip

The following proposition shows that \textproc{SampleSingleRecord} achieves
the desired goal.

\begin{proposition}
\label{prop:distributionSingleRecord} \thomasnote{Assume the  condition
\begin{eqnarray}\label{eq:star}
\sum_{i=1}^{d}{\mathbb{P}
(X(t_{i})>a\log(n_0+k)+C)} \le g_{n_0}(k)\quad\mbox{for $k\ge 1$.}
\end{eqnarray}} 
For $\tonnote{n} \ge n_0$,
if $(\widetilde X_1, \ldots, \widetilde X_{\widetilde T})$ has the
distribution of the output of \textproc{SampleSingleRecord} conditioned on
not being `degenerate,' then we have

\begin{enumerate}
\item the algorithm \textproc{SampleSingleRecord} returns 'degenerate' with
probability $\mathbb{P}(T_{\tonnote{n}}=\infty)$,

\item the length $\widetilde T$ has the same distribution as $T_{\tonnote{n}}$  given $T_{\tonnote{n}}< \infty$, and

\item the distribution of $(\widetilde X_1, \ldots, \widetilde X_{\widetilde
T})$ given $\widetilde T=\ell$ is the same as the distribution of $(X_1,\ldots,X_\ell)$ given $T_{\tonnote{n}}=\ell$.
\end{enumerate}
\end{proposition}

\begin{proof}
Write $A_m = \{x\in \mathbb{R}^d: \max_i x_i > a\log(\tonnote{n}+m)+C\}$ for $m\ge 1$%
. For $B_1\subset A_1^c,\ldots,B_{k-1}\subset A_{k-1}^c$ and $B_k\subset A_k$,
we have
\begin{eqnarray*}
\lefteqn{\mathbb{P}(\widetilde X_1\in B_1,\ldots, \widetilde X_{k-1}\in
B_{k-1},\widetilde X_k \in B_{k}, \widetilde T=k)} \\
&=& \mathbb{P}(K=k) \mathbb{P}( X\in B_1)\cdots \mathbb{P}( X\in
B_{k-1}) \\
&&\mbox{}\times \mathbb{P}^{(\tonnote{n}+k)}\left( U g_{n_0}(k) \leq \frac{d\mathbb{P%
}}{d\mathbb{P}^{(\tonnote{n}+k)}}(X), X \in B_{k}\right) \\
&=& \mathbb{P}(K=k) \mathbb{P}( X\in B_1)\cdots \mathbb{P}( X\in
B_{k-1}) \\
&&\mbox{}\times \mathbb{E}^{(\tonnote{n}+k)}\left(\frac1{g_{n_0}(k)} \,\frac{d%
\mathbb{P}}{d\mathbb{P}^{(\tonnote{n}+k)}}(X)  I(X \in B_{k})\right) \\
&=& g_{n_0}(k) \mathbb{P}( X\in B_1)\cdots \mathbb{P}( X \in
B_{k-1}) \frac{\mathbb{P}(X\in B_{k})}{g_{n_0}(k)} \\
&=& \mathbb{P}( X_1\in B_1,\ldots, X_{k}\in B_{k}, T_{\tonnote{n}}=k),
\end{eqnarray*}
and \ignore{both}\zhipengnote{all} claims follow from this identity. The second equality follows from
the assumption, which implies that $d\mathbb{P}/d\mathbb{P}^{(\tonnote{n}+k)}(x) /
g_{n_0}(k)$ is bounded by 1 for all $k\ge 1$ and $x\in \mathbb{R}^d$.
\end{proof}

\subsection*{Choosing $n_0$ and the density $g_{n_0}$}

We start with $g_{n_0}$, guided by \ignore{the condition in Proposition~\ref{prop:distributionSingleRecord}} \thomasnote{\eqref{eq:star}} 
and the requirement that we need to sample
from $g_{n_0}$. \tonnote{The random variable $K$ is a proxy for the first-record epoch $T_{n_0}$,
the distribution of which we can approximate with a union-bound. This leads to the idea to use} \ignore{A natural choice is}, for $k\ge 1$,
\begin{eqnarray}  \label{eq:defg}
g_{n_0}(k) = \frac{\int_{k-1}^k \phi((a\log(n_0+s)+C)/\overline\sigma)ds} {%
\int_0^\infty \phi((a\log(n_0+s)+C)/\overline\sigma)ds},
\end{eqnarray}
where $\phi(\cdot)$ is the density function of the standard normal
distribution, $\overline\sigma^2=\max_{t \in T}\mathrm{Var}%
(X(t))$. The following lemma resolves the sampling question.

\begin{lemma}
\label{lem:gsample} Let $U$ be a uniform random variable on $(0,1)$. The
quantity
\begin{equation*}
\left\lceil \exp\left\{\frac{\overline\sigma^2}{a^2}-\frac{C}{a}+\frac{%
\overline\sigma}a \overline{\Phi}^{-1}\left(U \,\overline{\Phi}\left(\frac{%
a\log n_0+C}{\overline\sigma}-\frac{\overline\sigma}a\right)\right)\right\}
-n_0 \right\rceil
\end{equation*}
\ignore{is a random variable with} \thomasnote{has} probability mass function $g_{n_0}$, where $%
\lceil\cdot\rceil$ is the round-up operator, \thomasnote{$\overline{\Phi}%
=1-\Phi$,} and $\overline{\Phi}^{-1}$ is the inverse of $%
\overline{\Phi}$.
\end{lemma}

\begin{proof}
Write $f_{n_0}(U)$ for the expression inside the exponential operator. For $%
k\ge 1$, we have
\begin{eqnarray*}
\mathbb{P}(\lceil \exp(f_{n_0}(U))-n_0\rceil\ge k) &=& \mathbb{P}(
f_{n_0}(U) > \log(n_0+k-1)) \\
&=& \frac{\overline{\Phi}\left({(a\log (n_0+k-1)+C)}/{\overline\sigma}-{%
\overline\sigma}/a\right)}{\overline{\Phi}\left({(a\log n_0+C)}/{%
\overline\sigma}-{\overline\sigma}/a\right)},
\end{eqnarray*}
so it remains to show that this equals
\begin{equation*}
\sum_{m\ge k}g_{n_0}(m) = \frac{\int_{n_0+k-1}^\infty \phi((a\log
x+C)/\overline\sigma)dx}{\int_{n_0}^\infty \phi((a\log
x+C)/\overline\sigma)dx}.
\end{equation*}
To see this, we note that, for $y>0$,
\begin{eqnarray}\label{eq:2star}
\int_y^\infty \phi((a\log(x)+C)/\overline\sigma) dx &=& \frac 1{\sqrt{2\pi}
}\int_{\log y}^\infty \exp\left(-\frac {(at+C)^2}{2\overline\sigma^2} + t
\right) dt\nonumber \\
&=& \frac{ e^{-C/a}}{\sqrt{2\pi}\phi(\overline\sigma/a)/(\overline\sigma/a)}
\times \overline{\Phi}((a\log y+C)/\overline\sigma-\overline\sigma/a)\\
&=& \thomasnote{r(y)\,,}\nonumber
\end{eqnarray}
and we thus obtain the claim.
\end{proof}

The next lemma shows that, for large enough $n_0$, \thomasnote{the choice of $g_{n_0}$ as in \eqref{eq:defg}}
ensures that \ignore{the condition from Proposition~\ref%
{prop:distributionSingleRecord}} \thomasnote{\eqref{eq:star}} 
is satisfied. \thomasnote{The lemma} also shows how $\mathbb{P}%
(T_{\tonnote{n}}<\infty)$ for $\tonnote{n} \ge n_0$ can be controlled explicitly.

\begin{proposition}
\label{prop:choiceL} \ignore{Let $\delta\in(0,1)$ be given.} 
If $n_0$ satisfies 
$a\log n_0 + C\ge \overline \sigma$ and \thomasnote{$d\,r(n_0)\le \delta$ for 
a given $\delta\in (0,1)$, } \ignore{
\begin{equation*}
{de^{-C/a}} {\overline\Phi\left(\frac{a\log n_0+C}{\overline\sigma}-\frac{%
\overline\sigma}a\right)}\le \delta \sqrt{2\pi}{\frac{\phi(\overline\sigma/a)%
}{\overline\sigma/a}},
\end{equation*}}
then \ignore{the condition in Proposition~\ref{prop:distributionSingleRecord}} \thomasnote{\eqref{eq:star}} is
satisfied and \textproc{SampleSingleRecord}$(a,C, \tonnote{n})$ returns
`degenerate' \thomasnote{at least} with probability $1-\delta$.
\end{proposition}

\begin{proof}
\thomasnote{Since  $\overline\Phi(x)\le \phi(x)$ for $x\ge 1$,  
$d\,r(n_0)\le \delta$, and in view
of \eqref{eq:2star}} we have
\begin{eqnarray}
\sum_{i=1}^d \mathbb{P}(X(t_i)>a\log(n_0+k)+C) &\le&
d\overline\Phi\left((a\log(n_0+k)+C)/\overline \sigma\right)\nonumber \\
&\le& d\phi\left((a\log(n_0+k)+C)/\overline \sigma\right) \nonumber\\
&\le& d \int_{k-1}^k \phi\left((a\log(n_0+s)+C)/\overline \sigma\right)ds \nonumber\\
&=& d \int_0^\infty \phi\left((a\log(n_0+s)+C)/\overline \sigma\right)ds\,
g_{n_0}(k) \nonumber\\
&=& \ignore{\frac{de^{-C/a}}{\sqrt{2\pi}} \frac{\overline\Phi((a\log
n_0+C)/\overline\sigma-\overline\sigma/a)}{\phi(\overline\sigma/a)/(%
\overline\sigma/a)}} \thomasnote{d\,r(n_0)\, g_{n_0}(k) <\delta\, g_{n_0}(k)\,.}\label{eq:3star}
\end{eqnarray}
\ignore{where the last equality is established in the proof of Lemma~\ref%
{lem:gsample}. By assumption, the factor in front of $g_{n_0}(k)$ is bounded
by $\delta$.} This proves the first claim.
\par
Applying Proposition~\ref{prop:distributionSingleRecord} and 
\thomasnote{\eqref{eq:3star} for every $k$,} \ignore{and summing the
preceding display over $k$, we find that} the probability that 
\textproc{SampleSingleRecord} does not return `degenerate' is \thomasnote{bounded as follows:}
\begin{eqnarray*}
\sum_{k=1}^\infty \mathbb{P}(T_{\tonnote{n}} = k) &\le& \sum_{k=1}^\infty
\sum_{i=1}^d \mathbb{P}(X(t_i)>a\log(\tonnote{n}+k)+C) \\
&\le& \sum_{k=1}^\infty \sum_{i=1}^d \mathbb{P}(X(t_i)>a\log(n_0+k)+C) \\
&\thomasnote{<}& \ignore{\frac{de^{-C/a}}{\sqrt{2\pi}} \frac{\overline\Phi((a\log
n_0+C)/\overline\sigma-\overline\sigma/a)}{\phi(\overline\sigma/a)/(%
\overline\sigma/a)}} \thomasnote{\delta\,\sum_{k=1}^\infty g_{n_0}(k)=\delta,}
\end{eqnarray*}
which proves the second claim.
\end{proof}

\subsection{Beyond $N_X$}

\label{sec:beyondNX} We next describe how to sample $(X_1,\ldots,X_n)$
conditionally on $T_{\tonnote{n}}=\infty$. As in Section~\ref{sec:beyondNS} we use
an acceptance/rejection algorithm, but we have to modify the procedure
slightly because we work with a sequence of iid random fields instead of a
random walk.

\bigskip

\textbf{Function} \textsc{SampleWithoutRecordX} $\left(\tonnote{n},\ell\right)$
\textbf{: Samples} $(X_1,\ldots,X_\ell)$ \textbf{conditionally on} $%
T_{\tonnote{n}}=\infty$ for $\ell\ge 1$

Step 1: Repeat

Step 2: \qquad $X \gets$ sample $(X_1,\ldots,X_\ell)$ under $\mathbb{P}$

Step 3: Until $\sup_{1\le k\le \ell} [X_k- a\log (\tonnote{n}+k)]<C$

Step 4: Return $X$

Step 5: EndFunction

\subsection{The full algorithm}

\label{sec:fullalgorithmX} We summarize our findings in this section in our
full algorithm for sampling $(X_1,\ldots, X_{N_X+\ell})$ under $\mathbb{P}$
given some $\ell \ge 0$.

The idea is to successively apply \textproc{SampleSingleRecord} to generate
the $\eta_i$ from the beginning of this section. Starting from $\eta_0=n_0$
satisfying the requirements in Proposition~\ref{prop:choiceL}, \tonnote{we generate $T_n$ where $n$ is}
replaced by each of the subsequent $\eta_i$. As a result, we have $\mathbb{P}%
(\eta_i=\infty|\eta_{i-1}<\infty) \ge 1-\delta$ by Proposition~\ref{prop:choiceL}. 
Thus, the number of records is bounded in probability by a
geometric random variable with parameter $1-\delta$.

\bigskip

\subparagraph{\textbf{Algorithm X: Samples }$(X_1,\ldots, X_{N_X+\ell})$
\textbf{given} $a \in (0,1]$, $\delta\in(0,1)$, $C\in\mathbb{R}$, $\overline%
\protect\sigma>0$, $\ell\ge 0$}

\label{alg:N_X}

\# {$n_0$ must satisfy the requirements in Proposition~\ref{prop:choiceL}. }

Step 1: $X\gets [\,]$, $\eta\gets n_0$

Step 2: $X \gets$ sample $(X_1,\ldots,X_\eta)$ under $\mathbb{P}$

Step 3: Repeat

Step 4: \qquad segment $\gets \textproc{SampleSingleRecord}(a,C,\ignore{n_0,} \eta)$

Step 5: \qquad If {segment is not `degenerate'}

Step 6: \qquad \qquad $X \gets [X, \text{segment}]$

Step 7: \qquad \qquad $\eta \gets \text{length}(X)$

Step 8: \qquad EndIf

Step 9: Until {segment is `degenerate'}

Step 10: If {$\ell>0$}

Step 11: \qquad $X\gets[X,\textproc{SampleWithoutRecordX}(\eta,\ell)]$

Step 12: EndIf

\section{Final Algorithm and Proof of Theorem~\protect\ref{Thm_Main}\label{sec:complexity}}

In this section, we give our final algorithm. We also \thomasnote{provide} 
the remaining arguments showing why \thomasnote{the algorithm outputs} \ignore{produces}exact samples and prove a bound on the
computational complexity. Together these proofs establish Theorem~\ref{Thm_Main}.
\par
We start with a description of our final algorithm for sampling $M$, which
exploits that for $S_n= \gamma n-A_n$ and $N_A=N_S$, we have $S_n<0$ and
therefore $A_n\ge \gamma n$ for $n>N_A$.

\bigskip

\subparagraph{\textbf{Algorithm M: Samples }$(M(t_1),\ldots,M(t_d))$ \textbf{%
given} $\protect\delta \in (0,1)$, $a\in(0,1]$, $\protect\gamma<\mathbb{E} A_1$, $C\in
\mathbb{R}$, $\overline \protect\sigma$}

\label{alg:final} ~\newline

Step 1: Sample $A_1,\ldots, A_{N_A}$ using Steps 1--9 from \textbf{Algorithm
S} with $S_n=\gamma n-A_n$.

Step 2: Sample $X_1,\ldots, X_{N_X}$ using Steps 1--9 from \textbf{Algorithm
X}.

Step 3: Calculate $N_a$ with (\ref{eq:defNmu}) and set $N=\max(N_A,N_X,N_a)$.

Step 4: If {$N>N_A$}

Step 5: \qquad Sample $A_{N_A+1},\ldots, A_N$ as in Step 10--12 from \textbf{%
Algorithm S} with $S_n=\gamma n-A_n$.

Step 6: EndIf

Step 7: If {$N>N_X$}

Step 8: \qquad Sample $X_{N_X+1},\ldots, X_N$ as in Step 10--12 from \textbf{%
Algorithm X}.

Step 9: EndIf

Step 10: Return $M(t_i) = \max_{1\le n\le N} \left\{-\log A_n + X_n(t_i) +
\mu(t_i)\right\}$ for $i=1,\ldots,d$.

\bigskip

%
%
%
%


{The pathwise construction in Section~\ref{sec:strategy} implied that the
output of Algorithm~\ref{alg:final} is an exact sample of $\{M(t_1),
\ldots , M(t_d)\}$. \thomasnote{Thus it remains to study the} running time of
\textbf{Algorithm M}.

\subsection{Computational complexity} 
We next study the truncation point $N$ in (\ref{eq:truncationsup}). Because 
\ignore{{$N_X$}}\zhipengnote{the number of records} is bounded in probability by a geometric random variable, it is clear
that $N<\infty$ almost surely. 

\ignore{It is our next} \thomasnote{Our aim is  to study the dependence
of our algorithm on the dimension $d$.}
\ignore{to show how our algorithm depends on $d$.} The only places
where $d$ enters the algorithm are in the definition of $n_0$ and the
measure $\mathbb{P}^{(n)}$. Sampling from the latter happens at most a
geometric number of times with parameter $1-\delta$, so the computational
complexity is dominated by the choice of $n_0$.

For any $\zeta>0$, if $d$ is large enough and 
\ignore{ignoring rounding} \thomasnote{if we ignore rounding, the following 
choice of $n_0=n_0(d)$}
\begin{equation*}
\log(n_0(d)) = \frac{\overline\sigma^2}{a^2} -\frac C a+ \frac{%
\overline\sigma}{a}\sqrt{(2+\zeta)\log\left(\frac {de^{-C/a}}{\delta\sqrt{%
2\pi}\phi(\overline\sigma/a)/(\overline\sigma/a)} \right)}
\end{equation*}
satisfies the assumption \thomasnote{$d\,r(n_0)\le \delta$} of Proposition~\ref{prop:choiceL}.

\ignore{The next lemma is useful to show} \thomasnote{The following result
will be needed  for the proof of} the second part of Theorem~\ref{Thm_Main}.
\thomasnote{Recall that $K$ is a positive integer-valued random variable 
with probability mass function $g_{n_0}$.}
\begin{lemma}
\thomasnote{For $ p \ge 1$,} \tonnote{we have} 
$\log(\mathbb{E}[K^p])=O(\log n_0)$ \tonnote{as $d\to\infty$}.
\end{lemma}

\begin{proof}
Assume $n_0$ sufficiently large. \thomasnote{Then}
\begin{eqnarray*}
\mathbb{E}[K^p]&=&\sum_{k=1}^\infty k^p g_{n_0}(k) \\
&\le& \frac{\int_0^\infty (s+n_0)^p \phi((a\log(n_0+s)+C)/\overline\sigma)ds%
} {\int_0^\infty \phi((a\log(n_0+s)+C)/\overline\sigma)ds} \\
&=& e^{\frac{p^2\overline \sigma}{2a^2}-\frac{Cp}{a}} \frac{\overline
\Phi((a\log(n_0)+C-p \overline\sigma^2/a)/\overline\sigma -\overline\sigma
/a)} {\overline \Phi((a\log(n_0)+C)/\overline\sigma-\overline\sigma/a)} \\
&\le& e^{\frac{p^2\overline \sigma}{2a^2}-\frac{Cp}{a}} \frac{\frac{1}{%
(a\log(n_0)+C-p \overline\sigma^2/a)/\overline\sigma -\overline\sigma /a}
\phi((a\log(n_0)+C-p \overline\sigma^2/a)/\overline\sigma -\overline\sigma
/a) }{\frac{(a\log(n_0)+C)/\overline\sigma -\overline\sigma /a}{%
((a\log(n_0)+C)/\overline\sigma -\overline\sigma /a)^2+1} \phi((a%
\log(n_0)+C)/\overline\sigma -\overline\sigma /a) } \\
&\le& 2 e^{\frac{p^2\overline \sigma}{2a^2}-\frac{Cp}{a}} \exp \left(-\frac{%
p^2 \overline\sigma^2}{2a^2}+\frac{p \overline \sigma}{a}(a\log(n_0)+C)/%
\overline\sigma -\overline\sigma /a)\right) \\
&=& 2 \exp\left(p \log(n_0)-\frac{p \overline\sigma^2}{a^2}\right)\,,
\end{eqnarray*}
\thomasnote{Therefore} $\log \mathbb{E}[K^p]\le p \log(n_0)+\log 2 -
p\overline\sigma^2/a^2$.
\end{proof}

\ignore{
After observing that $\log(n_0(d))=O(\sqrt{\log d})$, we are now ready 
to} \thomasnote{Next we show that} $\log \mathbb{E}[N_X^p]=O(\sqrt{\log d})$. 
We have \thomasnote{the decomposition}
\begin{equation*}
N_X=n_0+\sum_{i=1}^G K_i,
\end{equation*}
\thomasnote{where $K_i$ are iid copies of $K$} 
\ignore{generated from $g_{n_0}(\cdot)$}, \thomasnote{$G$ is the last time
that the segment is not `degenerate' and the definition of $n_0$ implies
 $\log(n_0(d))=O(\sqrt{\log d})$.} 

 Proposition~\ref{prop:choiceL} shows
that $G$ is bounded by a geometric random variable $G^{\prime }$ with
parameter $\delta$ almost surely, while $G^{\prime }$ is independent of the
sequence ignore{\color{magenta}@@@@}\thomasnote{$(K_i)$.} Therefore, we have \tonnote{by Jensen's inequality}
\begin{eqnarray*}
\mathbb{E}[N_X^p]&\le& \mathbb{E}\left[\left(n_0+\sum_{i=1}^{G^{\prime }}K_i\right)^p\right] \\
&\le& \tonnote{\mathbb{E}\left[\left(n_0^p+\sum_{i=1}^{G^{\prime }}K_i^p\right)(1+G^{\prime})^{p-1}\right]} \\
&=&
\tonnote{n_0^p\mathbb{E}\left[ (1+G^{\prime})^{p-1}\right]+\mathbb{E}[K_1^p] \mathbb{E}[G^\prime (1+G^{\prime})^{p-1}] . }
\end{eqnarray*}
Therefore we have shown that $\log \mathbb{E}[N_X^p]=O(\sqrt{\log d})$,
which means \tonnote{that} $\mathbb{E}[N_X^p]$ increases slower than $d^\epsilon$ for any $\epsilon>0$.

\thomasnote{Clearly, $N_A$ or $N_a$ do} not depend on $d$. We only need to show $\mathbb{%
E} [N_A^p]<\infty$, and $\mathbb{E} [N_a^p]<\infty$.

Recall that in Section~\ref{Section_Milestone_RW_new} we sample the
downcrossing segment of the random walk with the nominal distribution, then
the upcrossing segment with the exponential tilted distribution. We denote
the $i$'th downcrossing segment having length $\tau^-_i$, and the $i$'th
upcrossing segment having length $\tau^+_i$. Therefore,
\begin{equation*}
N_{A}=\sum_{i=1}^{L}(\tau^-_i+\tau^+_i),
\end{equation*}
where $L$ is the first time that the upcrossing segment is \tonnote{`degenerate'}.
Recall that $\tau^+$ denotes the first upcrossing time of level 0. Because for
any $x\le 0$,
\begin{equation*}
\mathbb{P}_x(\tau^+ = \infty) \ge \mathbb{P}_0(\tau^+ = \infty) >0,
\end{equation*}%
$L$ is a.s.~bounded by a geometric random variable $L^{\prime }$ with
parameter $q<1$\ignore{,  under \tonnote{the} exponential tilted distribution}.\newline

According to the discussion in Remark~\ref{rem:remove_assump}, we may assume
without loss of generality that $A_n$ has step sizes bounded by $r>0$.
Therefore, $S_{\tau^+_i}\le \gamma$, and $S_{\tau^-_i}\ge r$. Thus, with
Theorem 8.1 in \cite{Gutbook}, for any $p \ge 1$ and $\epsilon>0$, there
exists some constant $V>0$, such that
\begin{equation*}
\mathbb{E}[(\tau^-_{i})^{p(1+\epsilon)}] <V \qquad \text{and} \qquad \mathbb{%
E}[(\tau^+_{i})^{p(1+\epsilon)}] <V.
\end{equation*}
\thomasnote{Again using Jensen's inquuality,} we obtain
\begin{align*}
\mathbb{E}[N_{A}^p] &\le \mathbb{E }\left[\left(\sum_{i=1}^{L^{\prime
}}(\tau^-_i+\tau^+_i)\right)^p \right] \\
&\le \mathbb{E }\left[\frac{\sum_{i=1}^{L^{\prime
}}((\tau^-_i)^p+(\tau^+_i)^p)}{2L^{\prime }} (2L^{\prime})^p \right] \\
&\le \sum_{i=1}^{\infty }\mathbb{E }\left[((\tau^-_i)^p+(\tau^+_i)^p) %
 I(L^{\prime }\ge i) (2L^{\prime})^{p-1} \right] \\
&\zhipengnote{\le 2V^{\frac{1}{1+\epsilon}} \left(\mathbb{E }\left[(2L^{\prime})^{ (p-1)
\frac{1+\epsilon}{\epsilon}} L^{\prime }\right]\right)^{\frac{\epsilon}{1+\epsilon}}} 
<\infty.
\end{align*}

The value of $N_{a}$ is only required to satisfy \thomasnote{(see \eqref{eq:defNmu})}
\begin{equation}\label{eq:5star}
N_{a}\geq \left( \frac{A_{1}\,\exp (C-\underline{X}_{1})}{\gamma }\right) ^{%
\frac{1}{1-a}},
\end{equation}
\zhipengnote{while $a \in (0,1)$.} 
Therefore, for $a \in (0,1)$, we have 
\begin{equation*}
\mathbb{E}[N_{a}^{p}]=\left( \frac{\mathbb{E}[A_{1}]\exp (C)}{\gamma }\right) ^{\frac{p}{%
1-a}}\mathbb{E}[\exp (-\underline{X}_{1})^{p}]^{\frac{1}{1-a}} < \infty.
\end{equation*}%
This naturally holds by Assumption B2). \zhipengnote{When $a=1$, with proper choice of $C$, (\ref{eq:defNmu}) always holds.}

\subsection{Choosing $a$, $C$, and $\protect\gamma$}

\label{sec:aC} Although the \thomasnote{values of $a\in (0,1]$, $\gamma \in (0,\mathbb{E}[A_{1}]$) and $C\in \mathbb R$} do not affect the order of the computational
complexity of our algorithm, we are still interested in discussing some
guiding principles which can be used to choose those parameters for a
reasonably good implementation.
\par
First, note that among $N_{X},$ $N_{A},$ and $N_{a}$, only $N_{X}$ \zhipengnote{would
increase to $\infty$ as the number \thomasnote{$d$ of sampled 
locations} increases to $\infty$. 
(Although $N_a$ also increases in $d$, it remains bounded since $\underline{X}_{1}$ decreases to the minimum over $T$.)}  
Assuming that $C $ has been fixed, we can see that $N_{X}$ decreases pathwise while $a$
increases, therefore we should try to choose $a$ close to 1. On the other
hand, while $a \in (0,1)$, we have \eqref{eq:5star}.\ignore{
\begin{equation*}
N_{a}\geq \left( \frac{A_{1}\,\exp (C-\underline{X}_{1})}{\gamma }\right) ^{%
\frac{1}{1-a}}.
\end{equation*}}
If $A_{1}\,\exp (C-\underline{X}_{1})>\gamma $, then $N_{a}\nearrow \infty $
while $a\nearrow 1$. This analysis highlights a trade-off between the values
of $N_{X}$ and $N_{a}$ with respect to the choice of $a$. Because 
\thomasnote{${\mathbb E}[N_{X}]$}
is not explicitly tractable, we can have a reasonable balancing of the
computational effort by equating $n_{0}$ with \thomasnote{${\mathbb E}[N_{a}]$.} In particular, we
look for the largest value of $a\in (0,1)$ satisfying the following equation
\begin{equation}  \label{eqn:a}
\exp \left( \frac{\overline{\sigma }}{a}\overline{\Phi }^{-1}\left( \delta
\sqrt{2\pi }{\frac{\phi (\overline{\sigma }/a)}{d\overline{\sigma }/a}}%
\right) +\frac{\overline{\sigma }^{2}}{a^{2}}-\frac{C}{a}\right) =\mathbb{E}%
\left[ \left( \frac{A_{1}\,\exp (C-\underline{X}_{1})}{\gamma }\right) ^{%
\frac{1}{1-a}}\right] \thomasnote{.}
\end{equation}%
Note that the left-hand side converges to infinity as $a\searrow 0$ while
the right-hand side is bounded, but the right-hand side converges to
infinity as $a\nearrow 1$ while the left-hand side is bounded, so a solution
exists. Such a solution can be obtained by running a pilot run of $%
\underline X_1$, then search for the desired $a$ numerically.

Another approach consists \thomasnote{of} selecting $a=1$ and adjusting $C$ so that 
(\ref{eq:defNmu}) holds true for all $n\geq 1$. Therefore, we choose \thomasnote{
$
C=\underline{X}_{1}+\log \left(A_{1}/\gamma \right).
$}
The value of $C$ is random, but the algorithms can be modified accordingly,
\ignore{modifying} \thomasnote{by changing} the definition of $n_{0}$, which depends on $C$. However, the
expected computational cost 
\ignore{remains the same order as}\thomasnote{has the same order as in }the case when $C$ is
deterministic.

Similarly, $N_A$ increases pathwise while $\gamma$ increases, while $N_a$
decreases if $\gamma$ increases. One could get the empirical average value
of $N_A$ via simulation, and choose $\gamma$ accordingly such that $N_A$ and
$N_a$ are balanced.

\section{Tolerance Enforced Simulation\label{SECT_TES}}\setcounter{equation}{0}

In this section we illustrate a general procedure which can be applied so
that, for any given $\delta >0$ one can construct a fully simulatable
process \thomasnote{$M_{\delta }$,} with the property that
\begin{equation*}
\mathbb{P}\left( \sup_{t\in T}\left| M( t) -M_{\delta }(t) \right| \leq \delta \right) =1.
\end{equation*}
\thomasnote{For ease of notation we focus on the case $T=[0,1]$.}
The technique can be easily adapted to higher-dimensional sets $T$, \ignore{but the
important assumption involves}\tonnote{as long as one has} an infinite series representation for $X$
which satisfies certain regularity conditions. \ignore{ as we shall explain.}

\tonnote{A TES estimator can be used to easily obtain error bounds for sample-path functionals of the underlying field. For example, in the context of parametric catastrophe bonds, it is not uncommon to use the average extreme precipitation over a certain geographical region as the trigger\zhipengnote{; see \cite{MTAbond}}. This motivates estimating $\mathbb E[u ( \int_T M(s) ds )]$ for some 
function \thomasnote{be consistent: $u$} that is 
specified by the contract characteristics of the catastrophe bond.}
\ignore{ too many: For instance,}\thomasnote{If $u$} 
\tonnote{is Lipschitz continuous with Lipschitz constant 1, then one immediately obtains
\[
\left\vert\mathbb E\left[ u \left( \int_T M(s) ds \right) \right] -\mathbb E\left[ u \left( \int_T M_\epsilon(s) ds \right) \right]\right\vert \le |T| \epsilon.
\]
The form of the TES estimator discussed in this section has the feature that $\int_T M_\epsilon(s) ds$  can be evaluated in closed form. Thus, a TES estimator facilitates the error analysis that could otherwise be significantly more involved.}

The technique presented in this section is not limited to Gaussian processes, and we
do not make this assumption here.
As a result, we do not use Assumptions B1) and B2) in this section, but we replace them with \ignore{A)-D)}\zhipengnote{C1)-C4)} below.
However, Assumptions A1) and A2) on the
renewal sequence $( A_{n}) $ are in force throughout this section.

\subsection{An infinite series representation}
We assume that $(X_{n}(t))_{t\in T}$ can be expressed as an almost
surely convergent series of basis functions with random weights. We
illustrate the procedure with a particularly convenient family of basis
functions.

First, let us write any $m\ge 1$ as $m=2^{j}+k$ for $j\geq 0$ and
$0\leq k\leq 2^{j}-1$, and note that there is only one way to write $m$ in
this form. We assume that there exists a sequence of basis functions $(\Lambda_{m}\left( \cdot \right) )_{m\ge0}$, with support on $[0,1]$
(i.e., $\Lambda_{m}\left( t\right) =0$ for $t\not\in [0,1]$).
Moreover, we assume that $|\Lambda_{0}(t)|,|\Lambda_{1}(t)|\leq 1$
for all $t\in [0,1]$, and that for every $m\geq 1$,
\begin{equation*}
\Lambda _{m}(t)=\Lambda _{1}(2^{j}(t-k/2^{j})).
\end{equation*}%
\zhipengnote{In other words, for $m \ge 2$, each $\Lambda_m(\cdot)$ is a wavelet with the shape of $\Lambda_1(\cdot)$, while shrunk horizontally by factor of $2^j$, and shifted to start at $k/2^j$.}
\newline
We introduce normalizing constants, $\lambda _{0}>0$ and $\lambda
_{m}=\lambda ^{\prime }2^{-j\alpha }$ for $m\geq 1$, where $\alpha \in (0,1)$
and $\lambda ^{\prime }>0$. Finally, we assume that
\begin{equation*}
X_{n}(t)=\sum_{m=0}^{\infty }Z_{m,n}\Lambda_{m}(t)\lambda _{m},
\end{equation*}
where the random variables $(Z_{m,n})_{m\geq 0,n\geq 1}$ are iid. We shall
use $Z$ to denote a generic copy of the $Z_{m,n}$'s and we shall impose
suitable assumptions on the tail decay of $Z$. 
The parameter $\alpha$ relates to the H\"{o}lder
continuity exponent of the process $X_{n}$. For example, if $X_{n}$ is Brownian motion, $\alpha =1/2$. This interpretation
of $\alpha$ will not be used in our development, but it helps to provide intuition which can be used to \thomasnote{inform} the construction of a
model based on the basis functions that we consider. For more information on the connection to the H\"{o}lder properties
implied by $\alpha$, the reader should consult \cite{BCD} and the references therein.

Throughout, we use the following total order among the pairs $\{(m,n):m\geq 0, n\geq 1\}$. We say $( m,n) <( m^{\prime},n^{\prime }) $
if $m+n<m^{\prime }+n^{\prime }$ and in case $m+n=m^{\prime }+n^{\prime }$, we say that $( m,n) $ is smaller
than $\left( m^{\prime },n^{\prime }\right) $ in lexicographic order. In particular, we have
\begin{equation*}
( 0,1) <(0,2)<( 1,1) <( 0,3) < (1,2) < (2,1) <\cdots.
\end{equation*}
We let $\theta ( m,n) $ be the position of $\left( m,n\right) $ in the total order.
We also define $\eta (\cdot ) :\mathbb{N}\rightarrow \mathbb{N}\cup \{0\}\times \mathbb{N}$
to be the inverse function of $\theta ( \cdot ) $, and given $\theta\in \mathbb{N}$, we write
\[
\eta ( \theta) =( \eta_{m}(\theta) ,\eta_{n}( \theta) ).
\]

\subsection{Building blocks for our algorithm}
\label{sec:strategyTES}
We now proceed to describe the construction of $M_{\delta }$, which is adapted from a record-breaking technique introduced in \cite{BC2}. \thomasnote{An important building block of $M_\delta$ is}
\ignore{To define $M_\delta$, we use}the truncated series
\[
X_n(t;K) =\sum_{m\leq K} \lambda_m Z_{m,n}\Lambda_m(t).
\]
\tonnote{It is not required that $X_n$ agrees with the distribution of $X$ on dyadic points, 
although this \thomasnote{is} the case in our primary example of Brownian motion.}
We abuse notation by re-using notation such as $N_X$ and $N_A$ throughout our discussion of TES, but the random variables
are not the same as in the rest of the paper.
\par
Our algorithm relies on three random times.\ignore{ to be defined next.}\thomasnote{
We choose  suitable positive functions $a,\xi_0,\xi_1$ and a positive constant $\gamma$; see 
Proposition~\ref{Prop_core_tes} below for details.
\begin{enumerate}
\item $N_X$: for $k\ge N_X$ and $n\ge 1$,
\begin{equation}\label{eq:7star}
\sup_{t\in T} |X_{n}( t) -X_n(t;k)| \leq \xi_1(k) +\xi_0(k) a(n)
\end{equation}
and, for $n\ge N_X$,
\begin{equation}\label{eq:8star}
\sup_{t\in T} |X_n(t)| \le (a(0)\lambda_0+ \xi_{1}(1))+(\lambda_0+\xi_{0}(1))a(n)\,.
\end{equation}
\item
$N_A=N_A(\gamma)$: for $n\ge N_A$,
\begin{equation*}
A_n\ge \gamma \,n,
\end{equation*}
and we \ignore{can} sample $N_A$ jointly with $(A_1,\ldots,A_{N_A})$ 
using \textbf{Algorithm S} in Section~\ref{Section_Milestone_RW_new}.
\item
$N_\xi$: for $n\ge N_\xi$,
\begin{eqnarray}\label{eq:Nxi}
\lefteqn{(a(0)\lambda_0+ \xi_{1}(1))+(\lambda_0+\xi_{0}(1))a(n) - \log (n\gamma)}\nonumber\\ &\le& \inf_{t\in[0,1]} X_1(t,N_X) - \log(A_1) -\xi_1(N_X) - \xi_0(N_X) a(n).
\end{eqnarray}
\zhipengnote{We will choose $a$} \ignore{\thomasnote{Meaning momentarily?} }\zhipengnote{such that $N_\xi<\infty$ almost surely.}
\end{enumerate}}
Setting $N=\max(N_X,N_A,N_\xi)$, we have, for $t\in T$ and $n\ge N$,
\begin{eqnarray*}
-\log(A_n) + X_n(t) &\le& - \log A_n+(a(0)\lambda_0+ \xi_{1}(1))+(\lambda_0+\xi_{0}(1))a(n)\\
&\le& - \log (n\gamma)+(a(0)\lambda_0+ \xi_{1}(1))+(\lambda_0+\xi_{0}(1))a(n) \\
&\le&- \log(A_1)+ \inf_{t\in[0,1]} X_1(t,N_X) -\xi_1(N_X) - \xi_0(N_X) a(n) \\
&\le& - \log(A_1)+\inf_{t\in[0,1]} X_1(t)\\
&\le& -\log(A_1)+ X_1(t),
\end{eqnarray*}
and therefore, for $t\in T$,
\begin{equation}
\sup_{n\ge 1} \{-\log A_n + X_n(t) + \mu(t) \} = \max_{1\le n\le N} \{-\log A_n + X_n(t)+\mu(t)\}.
\end{equation}

If we select an integer $K_\delta\ge N_X$ such that $\xi_1(K_\delta) + \xi_0(K_\delta) a(n) \le \delta$, then
\[
M_\delta(t) = \max_{1\le n\le N} \{-\log A_n + X_n(t;K_\delta)+\mu(t)\}
\]
satisfies $\sup_{t\in T} |M(t)-M_\delta(t)|\le \delta$.

It remains to explain how to simulate $N_X$ jointly with $(X_1,\ldots, X_N)$ and how to construct $\xi_0$, and $\xi_1$.
For this, we use a variant of the record-breaking technique, but we first need to discuss our assumptions on the $Z_{m,n}$'s.

\subsection{Assumptions on the $Z_{m,n}$'s and an example}
We introduce some assumptions on the distribution of $Z$ in order to use our record-breaking algorithm.
We write $\overline{F}(\cdot )$ for the \thomasnote{right tail of the distribution} of $\left\vert Z\right\vert $,
that is $\overline{F}(t)=\mathbb{P}\left( \left\vert Z\right\vert >t\right) $ for $t\geq 0$.
Assume that we can find: a bounded and nonincreasing function $\overline{H}(\cdot )$ on $[0,\infty)$,
an easy-to-evaluate eventually nonincreasing function $\Gamma ( \cdot )$ on $\mathbb{N}$,
as well as some $\theta_0>0$, $b\in (0,1)$, and $\rho >0$ satisfying the following assumptions with
\[
a( n) =\rho (\log ( n+1) )^{b}:
\]

\begin{enumerate}
\item[C1)]
For $(m,n)$ satisfying $\theta(m,n)\ge \theta_0$, we have
$\overline{F}(a(m)+a(n))\leq \overline{H}(a(m))\overline{H}(a(n))$.

\item[C2)] We have
$\sum_{m=0}^{\infty }\overline{H}(a(m))<\infty$.

\item[C3)]
For $r>\theta_0$, we have
$1>\Gamma ( r) \geq \sum_{(m,n):\theta (m,n)>r}\overline{H}(a(m))\overline{H}(a(n))$.

\item[C4)]
We have $\sum_{r} r^{\varepsilon} \Gamma (r) <\infty $ for some $\varepsilon>0$.

\end{enumerate}

Assumptions C1), C2), and C3) are needed to run the algorithm, and Assumption C4) to bound moments of the
computational complexity.

As an example, we now show that these assumptions are satisfied \thomasnote{if $X_{n}$ is Brownian motion.}
Similar constructions are possible for fractional Brownian motion \thomasnote{(see \cite{AyTq}),} but we do not work out the details here.
First, $\Lambda _{0}( t)=tI( t\in [0,1]) $, $\Lambda _{1}( t)
=2tI( t\in [0,1/2]) +2(1-t)I( t\in (1/2,1]) $, $\alpha =1/2$, and $\lambda _{0}=\lambda ^{\prime }=1$; see
\cite{Sbook}. Second, the $Z_{m,n}$'s are iid
standard Gaussian random variables and one can select $\overline H(t)=\phi(t)$, the standard normal density,
so that we have Assumption A) for $\theta_0= \inf\{\theta: a(\eta_m(\theta))+a(\eta_n(\theta))\geq 2\sqrt{2\pi} \}$
and C2) is evident.
Moreover, selecting any $\rho>4$  and $b=1/2$ allows us to satisfy Assumptions C3) and C4). Indeed, note that
\begin{eqnarray*}
\lefteqn{\sum_{\theta (m,n)\ge r}\overline{H}(a(m))\overline{H}(a(n)) }\\
& =&\sum_{\theta (m,n)\ge r}\left( \frac{2}{\pi }\right) \exp \left( -\rho ^{2}%
\frac{\log (m+1)+\log (n+1)}{2}\right)  \\
& =&\sum_{\theta (m,n)\ge r}\left( \frac{1}{(m+1)(n+1)}\right) ^{\rho^{2}/2}\leq \sum_{\theta (m,n)\geq r}\left( \frac{1}{m+n}\right) ^{\rho^{2}/2}.
\end{eqnarray*}
The point $(m,n)$ with $\theta(m,n)=r$ is one of the $\ell(r)$ points
on the segment between $(\ell(r),0)$ and $(1,\ell(r)-1)$, where $\ell(r)=\lceil \sqrt{2r+1/4}-1/2\rceil$.
We therefore continue to bound as follows:
\[
\sum_{k\ge \ell(r)}k^{1-\rho ^{2}/2}\leq\int_{\ell(r)-1}^\infty x^{1-\rho ^{2}/2} dx =
\frac 1{\rho ^{2}/2-2} \left(\ell(r)-1\right)^{2-\rho ^{2}/2}.
\]
Thus, in the Brownian case we can define $\Gamma ( r)$ to be the right-hand side of the preceding display,
so for instance any $\rho >4$ implies Assumption C4).

\ignore{Let us discuss the implications of our strategy when applied to the Brownian
case.} \thomasnote{In the case when $X_{n}$ is standard}
Brownian motion we have
\begin{equation*}
\sum_{m=0}^{2^{r}-1}\lambda _{m}Z_{m,n}\Lambda _{m}( t)=X_{n}( t) ,
\end{equation*}%
for every dyadic point $t=j 2^{-r}$ with $j=0,1,\ldots,2^{r}$.
Therefore, once we fix any $\delta >0$ (say $\delta =1/2$), we can apply the
previous strategy to obtain $N$ and we can continue sampling $%
Z_{m,n}$ for $m\geq K_{\delta }$ if needed so that we can return
\begin{equation*}
M( t) =\max_{1\leq n\leq N}\left\{-\log \left( A_{n}\right)
+\sum_{m=0}^{2^{r}-1}\lambda _{m}Z_{m,n}\Lambda _{m}\left( t\right) \right\}.
\end{equation*}%
Consequently, we conclude that at least in the Brownian case the procedure
that we present here can be used to evaluate $\left\{ M\left( j/2^{r}\right)
\right\} _{j=0}^{d}$ with $d=2^{r}$ exactly and with expected computational
cost of order $O\left( d\cdot \mathbb{E}[ N] \right) =O\left(d\right) $ -- because 
$\mathbb{E}[ N] $ does not depend on $d$ \thomasnote{and is finite; see Theorem~\ref{Thm_TES} below.}

\subsection{Breaking records for the $Z_{m,n}$'s}
Define $T_{0}=0$, and, for $k\geq 1$,
\begin{equation*}
T_{k}=\inf \{\theta ( m,n) >T_{k-1}:\left\vert Z_{m,n} \right\vert >a(m)+a(n)\}.
\end{equation*}
In this subsection, given some integer $\theta_0\ge 0$,
we develop a technique to sample the random set $\mathcal T = \{T_{k}:T_{k}<\infty \}\cap\{\theta_0+1,\ldots\}$
jointly with $(Z_{m,n})_{m\ge 0, n\ge 1}$.
Indeed, given $\mathcal T$, the $Z_{m,n}$ are independent and have the following distributions.
For $\theta \left( m,n\right) \le \theta_0$, $Z_{m,n}$ has the nominal (unconditional) distribution.
For $\theta \left( m,n\right) \in \mathcal{T}$, $Z_{m,n}$ has the conditional distribution of $Z$ given $\{|Z|>a(m)+a(n)\}$,
and if $\theta \left( m,n\right) \notin \mathcal{T}$, $Z_{m,n}$ has the conditional distribution of $Z$ given
$\{|Z|\leq a(m)+a(n)\}$.

We first note that that only finitely many $T_{k}$'s are finite, so that we can once again apply a record breaking technique,
based on the record-breaking epochs $T_{k}$.
Indeed, applying Assumptions C1), we find that
\begin{equation*}
\sum_{m,n}P\left( \left\vert Z_{m,n}\right\vert >a(m)+a(n)\right) \leq
\sum_{m,n}\overline{H}( a(m)) \overline{H}( a(n)) =
\left(\sum_{m}\overline{H}( a(m)) \right)^{2}<\infty,  \label{eq_BC_TES}
\end{equation*}
and the claim follows from the Borel-Cantelli lemma.

The function \textsc{SampleRecordsZ} given below, which is directly adapted
from Algorithm 2w in \cite{BC2}, allows one to sequentially sample the elements
in $\mathcal{\{}T_{k}:T_{k}<\infty \mathcal{\}}$\ jointly with the $Z_{m,n}$%
's. The function \textsc{SampleRecordsZ} takes as input $\theta_0$ satisfying
$\Gamma ( \theta_0) <1$. \newline

\textbf{Function }\textsc{SampleRecordsZ}$\left( \theta_0\right) $\textbf{:
Samples the set $\mathcal{T}=\{T_{k}:T_{k}<\infty \} \cap \{\theta_0+1, \ldots\}$}


Step 1: Initialize $G\leftarrow \theta_0$ and $\mathcal{T}\leftarrow[\,]$.

Step 2: $u\leftarrow 1$, $d\leftarrow 0$. $V\leftarrow U(0,1)$.

Step 3: While $u>V>d$

Step 4: \qquad $G\leftarrow G+1$

Step 5: \qquad $d\leftarrow \max(d,(1-\Gamma ( G) )\times u)$

Step 6: \qquad $u\leftarrow \mathbb{P}(|Z|\leq a(\eta_m(G)) +a( \eta_n(G) ) )\times u$

Step 7: EndWhile

Step 8: If $V\geq u$, then $\mathcal{T}\leftarrow[\mathcal{T},G]$ and go to Step 2.

Step 9: If $V\leq d$, stop and return $\mathcal{T}$.\newline

The next proposition establishes that the output of the function \textsc{SampleRecordsZ}
has the desired distribution.

\begin{proposition}
\label{prop:recordsZ}
The output from \textsc{SampleRecordsZ}$\left(\theta_0\right) $ is a sample of
the set $\mathcal T =\{T_{k}:T_{k}<\infty\}\cap \{\theta_0+1,\ldots\}$.
Moreover, we have $\mathbb{E}\left[(\max(0,\sup\mathcal T))^\beta\right]<\infty$ for some $\beta>1$.
\end{proposition}

\begin{proof}
For simplicity we assume throughout this proof that $\theta_0=0$. For the first claim it suffices to show that
\textsc{SampleRecordsZ}$\left( 0\right) $ returns 
{$\mathcal T=\{T_{k}:T_{k}<\infty \} \cap \{1,2,\ldots\}$} without bias\ignore{;
subsequent $T_{k}$'s follow by the same reasoning}. We write $T=T_1$.

In Steps 3 through 5 the algorithm iteratively constructs the sequences $(u_j)$ and $(d_j)$ given by
\begin{equation*}
u_{j}=u_{j-1} \,\mathbb{P}(|Z|\leq a(\eta_m(j)) + a(\eta_n(j))),\quad d_{j}=\max(d_{j-1},u_{j-1}( 1-\Gamma (j) ))
\end{equation*}
with $u_{0}=1$ and $d_{0}=0$. It is evident that both sequences are monotone.
Moreover, we have $u_j = \mathbb{P}(T>j)$ for $j\geq 0$ and $\lim_{j\to\infty} u_j =\mathbb{P}( T=\infty )$.
Similarly, because $\lim_{j\to\infty} \Gamma(j)=0$ we obtain $\lim_{j\to\infty} d_j = \mathbb{P}(T=\infty)$.

Let $n(V)$ be the number of times Step 3
is executed before either going to Step 8 or Step 9. It suffices to check
that when Step 8 is executed then the element added to $\mathcal{T}$ has the
law of $T$ given $T<\infty $, and that Step 9 is executed with probability $\mathbb{P}(T=\infty)$.
For the former, we note that by definition of $n(V)$ and because $u_j\in(d_{j-1},u_j)$, we have for $j\ge 1$
\begin{eqnarray*}
\mathbb{P}\left( n(V)=j\,|\, V\geq u_{n(V)}\right) &=&\frac{\mathbb{P}\left( V\in
(d_{j-1},u_{j-1}),V\geq u_j\right) }{\mathbb{P}\left( V\geq u_{n(V)}\right) } \\
&=&\frac{\mathbb{P}\left( V\in (u_{j},u_{j-1})\right) }{\mathbb{P}\left( V\geq u_{n(V)}\right) }\\
&=&\frac{u_{j-1}-u_{j}}{1-\lim_{k\to\infty}u_k },
\end{eqnarray*}
which equals $\mathbb{P}( T=j|T<\infty)$ as desired.
For the latter, we note that
\[
\mathbb{P}(V\le d_{n(V)}) = 1-\mathbb{P}(V\ge u_{n(V)}) = \mathbb{P}(T=\infty).
\]

In preparation for the proof of the second claim of the proposition,
we bound the probability that the while loop requires more than $k\ge 1$ iterations:
\begin{eqnarray*}
\mathbb{P}(n(V)>k) &\le& \mathbb{P}(V\in (d_k,u_k)) \le \mathbb{P}(V\in (d_k,u_{k-1})) = u_{k-1}-d_k \\
&=& u_{k-1}-\max\{d_{k-1},u_{k-1}(1-\Gamma(k))\} \le \Gamma (k).
\end{eqnarray*}
As a consequence of the inequality
\[
\mathbb{E}[n(V)^\beta] = \sum_{k=0}^\infty ((k+1)^\beta - k^\beta) \mathbb{P}(n(V)>k)
\le 1+\sum_{k=1}^\infty ((k+1)^\beta - k^\beta) \Gamma(k),
\]
we find that $\mathbb{E}[n(V)^\beta]<\infty$ if $\sum_k k^{\beta-1} \Gamma(k)<\infty$.

We have a similar finite-moment bound for subsequent calls to the while loop.
Writing $n_i(V_1,\ldots,V_i)$ for the number of iterations in the $i$-th execution of the while loop,
where $V_1,V_2,\ldots$ are the iid standard uniform random variables generated in subsequent calls to Step 2.
Compared to the above argument for $i=1$, this quantity only depends on $V_1,\ldots,V_{i-1}$ through a random shift of $\Gamma$.
Because $\Gamma$ is eventually nonincreasing, there exists a constant $c'$ such that, for all $i\ge 1$,
\begin{equation}
\label{eq:evdecrbound}
\mathbb{E}[\zhipengnote{n_i(V_1,\ldots,V_i)^\beta}|V_1,\ldots,V_{i-1}] \le c' \sum_k k^{\beta-1} \Gamma(k).
\end{equation}

To prove a bound on the moment of $\sup\mathcal T$, we
first let $\Upsilon$ be the number of times we execute the while loop.
We then note that,
for any random variable $G$ and any $\beta\ge 1$, by Jensen's inequality,
\begin{eqnarray*}
\max(0,\sup \mathcal{T})^\beta&=&\left(\sum_{i=1}^{\Upsilon -1} n_i(V_1,\ldots,V_i)\right)^\beta
=\left(\sum_{i=1}^{\infty} n_i(V_1,\ldots,V_i) I\left( \Upsilon> i\right)\right)^\beta \\
&\le& \sum_{i=1}^\infty \left(\frac{n_i(V_1,\ldots,V_i)I(\Upsilon>i)}{\mathbb{P}(G=i)}\right)^\beta \mathbb{P}(G=i)\\
&=& \sum_{i=1}^\infty n_i(V_1,\ldots,V_i)^\beta I(\Upsilon>i) \mathbb{P}(G=i)^{1-\beta},
\end{eqnarray*}
because the right-hand side is finite almost surely.

Because the event $\{\Upsilon>i-1\}$ only depends on $V_1,\ldots,V_{i-1}$, we have by (\ref{eq:evdecrbound}),
\begin{eqnarray*}
\mathbb{E}[n_i(V_1,\ldots,V_i)^\beta I( \Upsilon>i)] &\le& \mathbb{E}[n_i(V_1,\ldots,V_i)^\beta I( \Upsilon>i-1)]\\
&=& \mathbb{E}\left[ I(\Upsilon>i-1) \mathbb{E}[\zhipengnote{n_i(V_1,\ldots,V_i)^\beta}|V_1,\ldots,V_{i-1}]\right]\\
&\le& c' \left(\sum_k k^{\beta-1} \Gamma(k) \right) P(\Upsilon>i-1)\\
&\le & c' \left(\sum_k k^{\beta-1} \Gamma(k) \right) P(T_1<\infty)^{i-1},
\end{eqnarray*}
where we use the fact that $\Upsilon $ is stochastically dominated by a geometric random
variable with success parameter $P( T_{1}=\infty) >0$.
Combining the preceding displays, we deduce that, for $\beta\ge 1$,
\[
\mathbb{E}\left[\max(0,\sup \mathcal{T})^\beta\right] \le c' \left(\sum_k k^{\beta-1} \Gamma(k) \right) \sum_{i=1}^\infty P(T_1<\infty)^{i-1}\mathbb{P}(G=i)^{1-\beta},
\]
which is seen to be finite for some $\beta>1$ by Assumption C4) upon choosing $G$ geometric with a suitably chosen success probability.
\end{proof}

\subsection{Truncation error of the infinite series}
We next write, for $k\ge 0$
\begin{equation*}
X_{n}( t) =X_{n}( t;k) +\sum_{m>k}\lambda_{m}Z_{m,n}\Lambda _{m}( t).
\end{equation*}
and it is our objective to study the truncation error, i.e., the second term.

The next proposition controls the truncation error in terms of functions
$\xi _{0}$ and $\xi _{1}$ defined for $r\ge 1$ through
\begin{eqnarray*}
\xi _{0}( r) &=&\lambda ^{\prime }(1-2^{-\alpha })^{-1}2^{-\alpha\lfloor \log _{2}( r) \rfloor}, \\
\xi_1(r) &=& \frac{\rho}{\log_2(e)} \left(\lfloor \log _{2}( r) \rfloor + \frac{2^{-\alpha}}{1-2^{-\alpha}}+2\right) \xi_0(r).
\end{eqnarray*}
Note that $\xi_{0}( r) ,\xi_{1}( r) \rightarrow 0$ as $r\rightarrow \infty $.
We also write
\[
N_{X}=\max \{\sup \mathcal{T},\theta_0-1\}.
\]
If $\mathcal{T}$
is empty then $\sup \mathcal{T=-\infty }$ and therefore $N_{X}=\theta_0-1$;
otherwise, if $\mathcal{T}$ is non-empty, then $\sup \mathcal{T}%
\geq \theta_0$ and therefore $N_{X}\geq \theta_0$.

\thomasnote{
\begin{proposition}
\label{Prop_core_tes}
For all $k\geq N_{X}$ and $n\ge 1$, we have \eqref{eq:7star}, and for all $n\geq N_{X}$, \eqref{eq:8star}.
\end{proposition}}

\begin{proof}\thomasnote{We observe that}
\[
\left\vert X_{n}\left( t\right) -X_{n}\left( t;\thomasnote{k}\right) \right\vert \leq
\sum_{m>\thomasnote{k}}\lambda _{m}a(m)|\Lambda _{m}(t)|+a(n)\sum_{m>\thomasnote{k}}\lambda _{m}|\Lambda_{m}(t)|.  \label{eq_diff_sums}
\]
If \thomasnote{$m> k\geq N_{X}$,} because $\theta (m,n)\geq m$, we have from the definition of $N_{X}$, that
\begin{equation*}
\left\vert \lambda _{m}Z_{m,n}\Lambda _{m}\left( t\right) \right\vert \leq
\lambda _{m}\left( a(m)+a(n)\right) |\Lambda _{m}(t)|.
\end{equation*}
We bound the summand of the second sum by noting that, for $r\ge 1$,
\begin{equation*}
\sup_{t\in T}
\sum_{m= r}^\infty\lambda _{m}|\Lambda _{m}(t)|\leq \sup_{t\in T} \sum_{j=\left\lfloor \log
_{2}(r)\right\rfloor }^\infty \sum_{k=0}^{2^{j}-1}\lambda _{2^{j}}|\Lambda_{2^{j}+k}(t)|\leq \sum_{j=\left\lfloor \log _{2}(r)\right\rfloor}^\infty \lambda ^{\prime }2^{-\alpha j}=\xi_{0}( r).
\end{equation*}%
We similarly bound the summand in the first sum, using the definition of $a(\cdot)$ and the fact that
\[
\sum_{j=k}^\infty j s^j=s^k\, \frac{k(1-s)+s}{(1-s)^2}
\]
for $|s|<1$. \thomasnote{These bounds establish \eqref{eq:7star}.}

\thomasnote{Now we turn to the proof of \eqref{eq:8star}. For $n\geq N_{X}$,}
\begin{equation*}
|X_n(t)|\le
\sum_{m=0}^\infty \left\vert \lambda _{m}Z_{m,n}\Lambda _{m}( t) \right\vert \leq
(a(0)+a(n))\lambda_0+
\sum_{m=1}^\infty \lambda _{m}\left( a(m)+a(n)\right) |\Lambda_{m}(t)|,
\end{equation*}
because $\theta (m,n)\geq N_{X}$ for each $n\geq N_{X}$.
The sum over $m$ is bounded by $\xi_1(1)+\xi_0(1)a(n)$ as shown in the proof \thomasnote{of \eqref{eq:7star}.}
\end{proof}

\subsection{Construction of $M_{\delta }$}

Now we are ready to provide the final algorithm for computing $M_{\delta }$.

\subparagraph{\textbf{Algorithm TES: Samples $M_{\delta }$ given $\delta >0$.}}~

Step 1: $\ \mathcal{T}\leftarrow $ Sample \textsc{SampleRecordsZ}$(\theta_0)$

Step 2: \ $N_{X}\leftarrow \max \{\sup \mathcal{T},\theta_0-1\}$

Step 3: \ Sample $Z_{m,n}$ from the nominal distribution if $\theta \left(
m,n\right) \leq \theta_0$

Step 4: \ For $0\leq m\leq N_{X}$ and $\theta \left( m,1\right)
>\theta_0$

Step 5: \ \qquad If $\theta ( m,1) \in \mathcal{T}$: sample $Z_{m,1}$ from the law of $Z$ given $\{|Z|>a(m)+a(1)\}$

Step 6: \ \qquad Else If: sample $Z_{m,1}$ from the law of $Z$ given $\{|Z|\leq a(m)+a(1)\}$

Step 7: \ EndFor

Step 8: \ Sample $A_{1},\ldots ,A_{N_{A}}$ using Steps 1--8 from {\bf Algorithm~S} with $S_{n}=\gamma n-A_{n}$.

Step 9: \ Compute $N_{\xi }$, the smallest $n$ for which (\ref{eq:Nxi}) holds, and let
$N\leftarrow \max (N_{X},N_{A},N_{\xi })$

Step 10: Sample $A_{N_{A}+1},\ldots ,A_{N}$ as in Step 10 from {\bf Algorithm~S} with $S_{n}=\gamma n-A_{n}$.

Step 11: Compute the smallest $K_\delta\geq N_{X}$ such that $\xi_{1}( K_\delta) +\xi _{0}( K_\delta) a(N)\leq \delta$.

Step 12: For $2\leq n\leq N$, $0\leq m\leq K_\delta$, $\theta (m,n) >\theta_0$ \ignore{OR} \zhipengnote{and also for}
$n=1$, $N_{X}<m\leq K_\delta$, $\theta (m,n) >\theta_0$

Step 13: \ \qquad If $\theta ( m,n) \in \mathcal{T}$: sample
$Z_{m,n}$ from the law of $Z$ given $\{|Z|>a(m)+a(n)\}$

Step 14: \ \qquad Else: sample $Z_{m,n}$ from the law of $Z$ given $\{|Z|\leq a(m)+a(n)\}$

Step 15: EndFor

Step 16: Return $M_{\delta }( t) =\max \{X_{n}(t;K_\delta) -\log ( A_{n}) \}$.

\subsection{Exponential moments of $\sup_{t\in[0,1]} |X(t)|$}
We need a bound on the exponential moments of $\sup_{t\in[0,1]} |X(t)|$
in order to analyze $N_{\xi }$. If $X$ is Gaussian and continuous, then
such a bound immediately follows from Borell's inequality \cite[Thm.~2.1.1]{AT}.
The following proposition establishes the existence of exponential moments in the generality
of the present section.

\begin{proposition}
\label{prop:expmoments}
For any $p>0$, we have
\begin{equation*}
\mathbb{E}\exp \left( p \sup_{t\in [0,1]}\left\vert X( t)\right\vert \right) <\infty .
\end{equation*}%
\end{proposition}
\begin{proof}
We first note that
\begin{equation*}
\sup_{t\in [ 0,1]}\left\vert X_{n}( t) \right\vert \leq \lambda_0 Z_{0,n} + \sum_{j=1}^{\infty }\lambda ^{\prime }
2^{-\alpha j}\max_{k=0,\ldots,2^{j}-1}\left\vert Z_{2^{j}+k,n}\right\vert.
\end{equation*}
It suffices to prove that the tail of the infinite sum in this expression is ultimately lighter than any exponential.
A union bound leads to, for $y\ge 0$,
\begin{eqnarray*}
\lefteqn{
\mathbb{P}\left( \sum_{j=1}^{\infty }\lambda ^{\prime }2^{-\alpha
j}\max_{k=0,\ldots,2^{j}-1}\left\vert Z_{2^{j}+k,n}\right\vert >y\right)}\\ &\leq
&\sum_{j=1}^{\infty }\mathbb{P}\left( \lambda ^{\prime }
2^{-\alpha j}\max_{k=0,\ldots,2^{j}-1}\left\vert Z_{2^{j}+k,n}\right\vert >
(2^{\alpha/2}-1) 2^{-\alpha j/2} y\right)  \label{eq:B_s} \\
&\leq &\sum_{j=1}^{\infty }\mathbb{P}\left( \max_{k=0,\ldots,2^{j}-1}\left\vert
Z_{2^{j}+k,n}\right\vert >\frac{(2^{\alpha/2}-1)2^{\alpha j/2}}{\lambda^\prime} y\right).
\end{eqnarray*}
\tonnote{Assumptions C1) and C2) imply that
$C^{\prime }:=\mathbb{E}\exp \left(  \left\vert Z/\rho\right\vert ^{1/b}\right)<\infty $ and therefore we have by Markov's inequality, for $t\ge 0$,}
\begin{equation*}
\mathbb{P}\left( \max_{k=0,\ldots,2^j-1} \left\vert Z_{2^j+k,n}\right\vert >2^{\alpha j/2} t\right) \leq
2^j \mathbb{P}\left( \left\vert Z\right\vert >2^{\alpha j/2} t\right) \leq C^\prime 2^j e^{-(t 2^{\alpha j/2}/\rho)^{1/b}}.
\end{equation*}
Select some $t_0>0$ and $\kappa\in(1,1/b)$
such that $(t 2^{\alpha j/2}/\rho)^{1/b}\ge j+t^\kappa$ for all $j\ge 1$ and $t\ge t_0$.
Using this bound results in a tail estimate that is summable over $j$ and lighter than any exponential distribution.
\end{proof}

\subsection{Complexity analysis}

We conclude this section with the following result which summarizes the
performance guarantee of Algorithm TES.
\tonnote{Higher moment bounds on the computational costs}
are readily found using the same arguments and a stronger version of Assumption C4).

\begin{theorem}
\label{Thm_TES} \thomasnote{Assume that the conditions {\rm A1), A2), {\rm C1)--C4)}} are in force.}
Given $\delta \in ( 0,1) $,
the output $(M_\delta(t))_{t\in T}$ of {\bf Algorithm TES} satisfies
\begin{equation*}
\sup_{t\in T} \left\vert M_{\delta }( t) -M( t) \right\vert \leq \delta.
\end{equation*}
Moreover, we have
\begin{equation*}
\mathbb{E}[ K_{\delta }] =O\left( (\delta /\log (1/\delta))^{-1/\alpha}\right),
\end{equation*}
where $\alpha$ is determined by the series representation of $X$.
Finally, the total computational costs of running Algorithm TES has 
\ignore{has finite expectation and is}\tonnote{expectation} at most $O\left( (\delta /\log (1/\delta))^{-1/\alpha}\right)$.
\end{theorem}

\begin{proof}
The first claim follows by construction, see Section~\ref{sec:strategyTES}.

\thomasnote{From Proposition~\ref{prop:recordsZ}
we have $\mathbb{E}[N_{X}^\beta] <\infty $} for some $\beta>1$. 
In order to analyze $N_{\xi }$, we use Proposition~\ref{prop:expmoments}. \zhipengnote{In fact, $N_\xi$ only has to be sufficiently large so that we have
\[
(\lambda_0+\xi_0(1)+\xi_0(N_X))\rho(\log (n+1))^b<\frac{1}{2} \log n
\]
and 
\[
-\frac{1}{2} \log n \le \inf_t X_1(t, N_X)-\log A_1-a(0) \lambda_0-\xi_1(1)-\xi_1(N_X)+\log \gamma
\]
for any $n \ge N_\xi$.}
With simple calculations, it follows from Proposition~\ref{prop:expmoments}
and Assumption A1) that \thomasnote{$\mathbb{E}[ N_{\xi}^p] <\infty $} for every $p>0$.
We have argued \thomasnote{in Section \ref{sec:complexity}} that \thomasnote{$\mathbb{E}[N_{A}^{p}] <\infty $}, so
we conclude that \thomasnote{$\mathbb{E}[ N^\beta] <\infty$.}
Finally, using the definition of $\xi_0(r) $ and $\xi_1(r)$ we can see that it there is a constant $\kappa>0$ such that
\begin{equation*}
K_{\delta }=O\left( \left[\frac{\delta}{(\log N)^b+ \kappa\log(1/\delta)}\right]^{-1/\alpha} \right).
\end{equation*}%
This leads to the bound on the first moment of $K_\delta$.
The expected running time of the algorithm is order $\mathbb{E}[K_\delta \times N]$, which \ignore{has a finite first moment} \tonnote{is finite} because $\mathbb{E}[N^\beta]<\infty$.
The complexity bound follows.
\end{proof}


\section{Numerical Results\label{SEC_NUMERICS}}\setcounter{equation}{0}

In this section we show some simulation results to empirically validate
\textbf{Algorithm M}. We also compare numerically the computational cost of
our \zhipengnote{record-breaking} method, noted as \zhipengnote{RB} in the following charts, 
with {\color{black} the existing} \ignore{\color{red}the}{\color{black}exact sampling algorithm developed in \cite{DM} by Dieker and Mikosch (DM) and the exact simulation algorithm using \thomasnote{extremal} function proposed in \cite{DEO} (EF)}\ignore{\color{red}algorithm proposed in \cite{DM}}. We
implemented all three algorithms in Matlab. For our algorithm, we \thomasnote{choose 
the values} of $a$ and $C$
according to our discussion in Section~\ref{sec:aC}. We let $C=0$, then
\thomasnote{choose} the largest $a\in (0,1)$ such that \eqref{eqn:a} holds.

We generated the Brown-Resnick processes, $M(t)=\sup\limits_{n\geq 1}\{-\log
A_{n}+X_{n}(t)-\sigma^2(t)/2\}$, on compact sets. If $X$ is a Brownian
motion it was shown in \cite{BR} that \thomasnote{$M$} has a stationary sample
path on $[0,1]$. Figure~\ref{fig:BM} shows sample paths of \thomasnote{$M$} in
this case. \thomasnote{In my printed version one can hardly see anything in this figure.
It seems to be too dark.}
\begin{figure}[tbp]
\centering
\begin{subfigure}{.5\textwidth}
  \centering
  \includegraphics[width=1\linewidth]{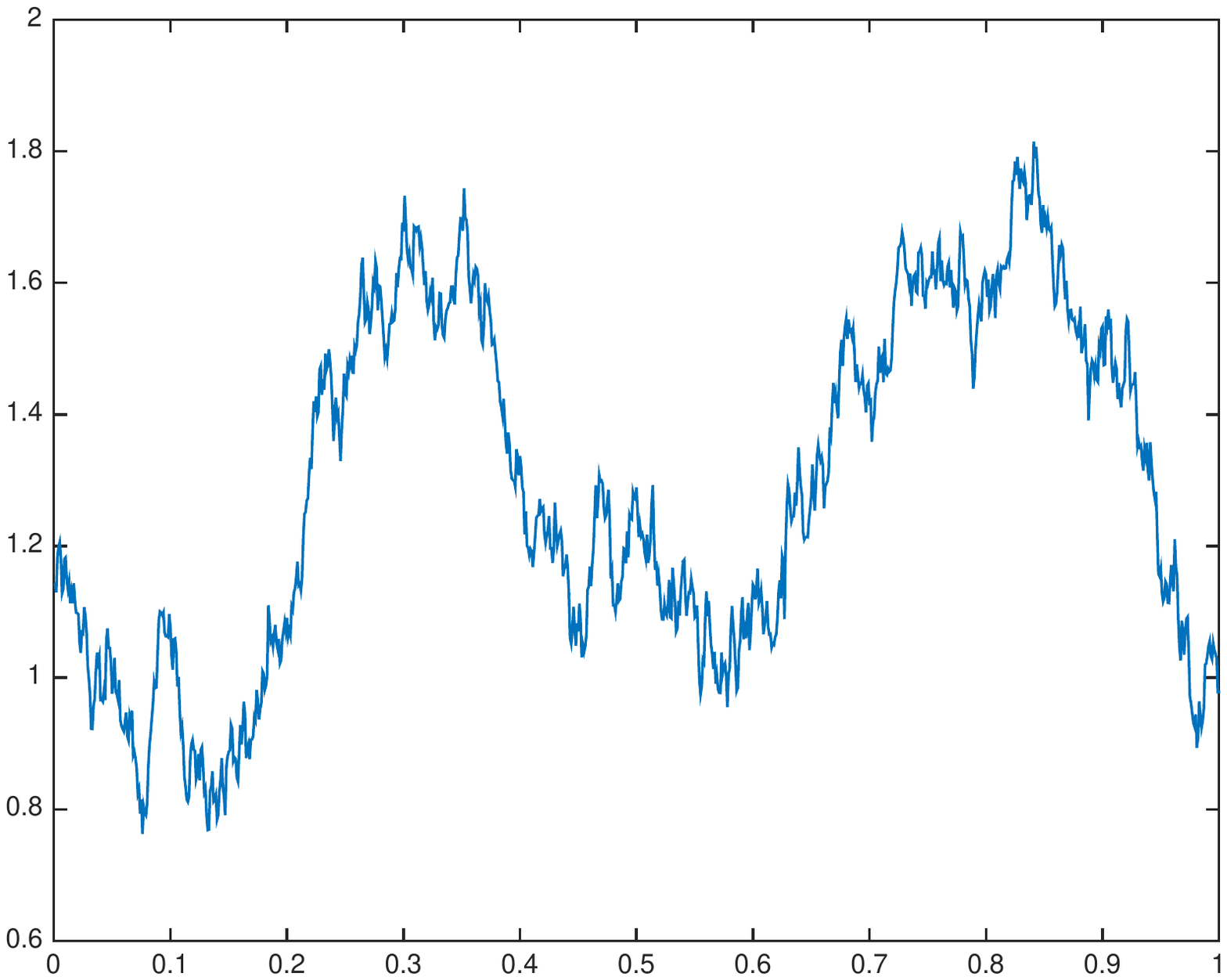}
\end{subfigure}%
\begin{subfigure}{.5\textwidth}
  \centering
  \includegraphics[width=1\linewidth]{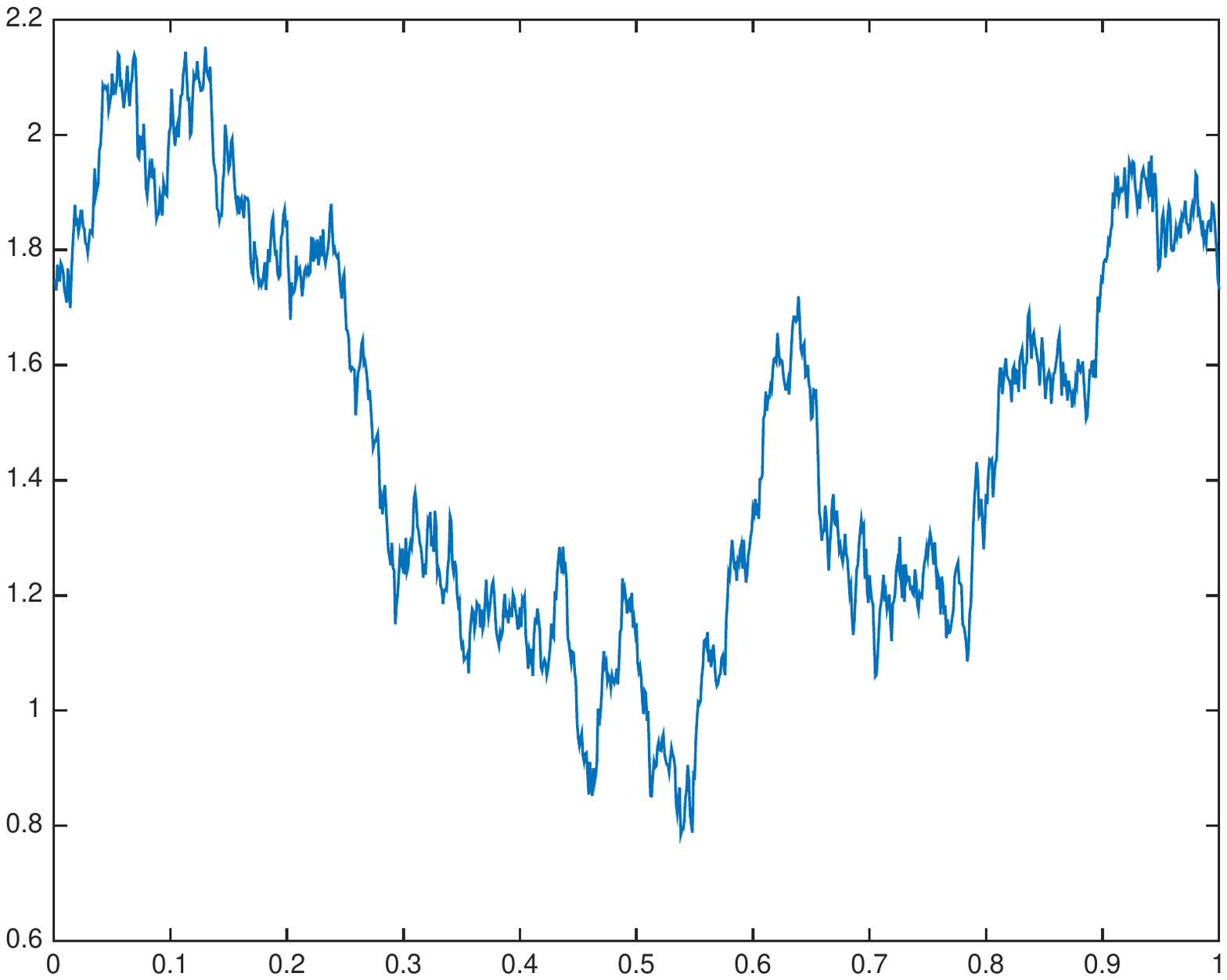}
\end{subfigure}
\caption{The Brown-Resnick process on $[0,1]$ with Brownian motion input.
The grid mesh is 0.001.}
\label{fig:BM}
\end{figure}
Figure~\ref{fig:BRS} presents two samples of the Brown-Resnick random field
on $[0,1]^{2}$ when $X$ is a Brownian sheet.
\begin{figure}[tbp]
\centering
\begin{subfigure}{.5\textwidth}
  \centering
  \includegraphics[width=1\linewidth]{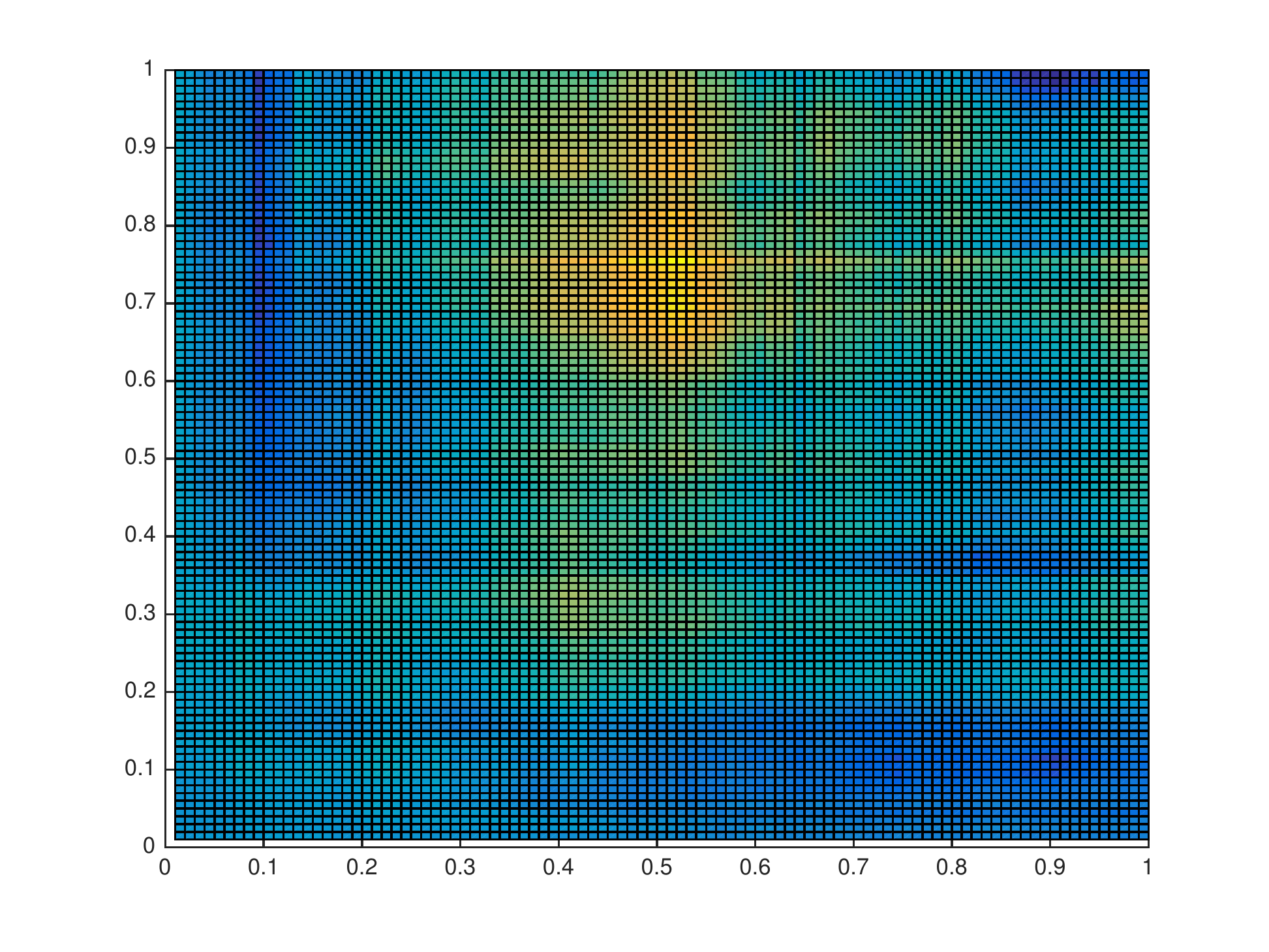}
\end{subfigure}%
\begin{subfigure}{.5\textwidth}
  \centering
  \includegraphics[width=1\linewidth]{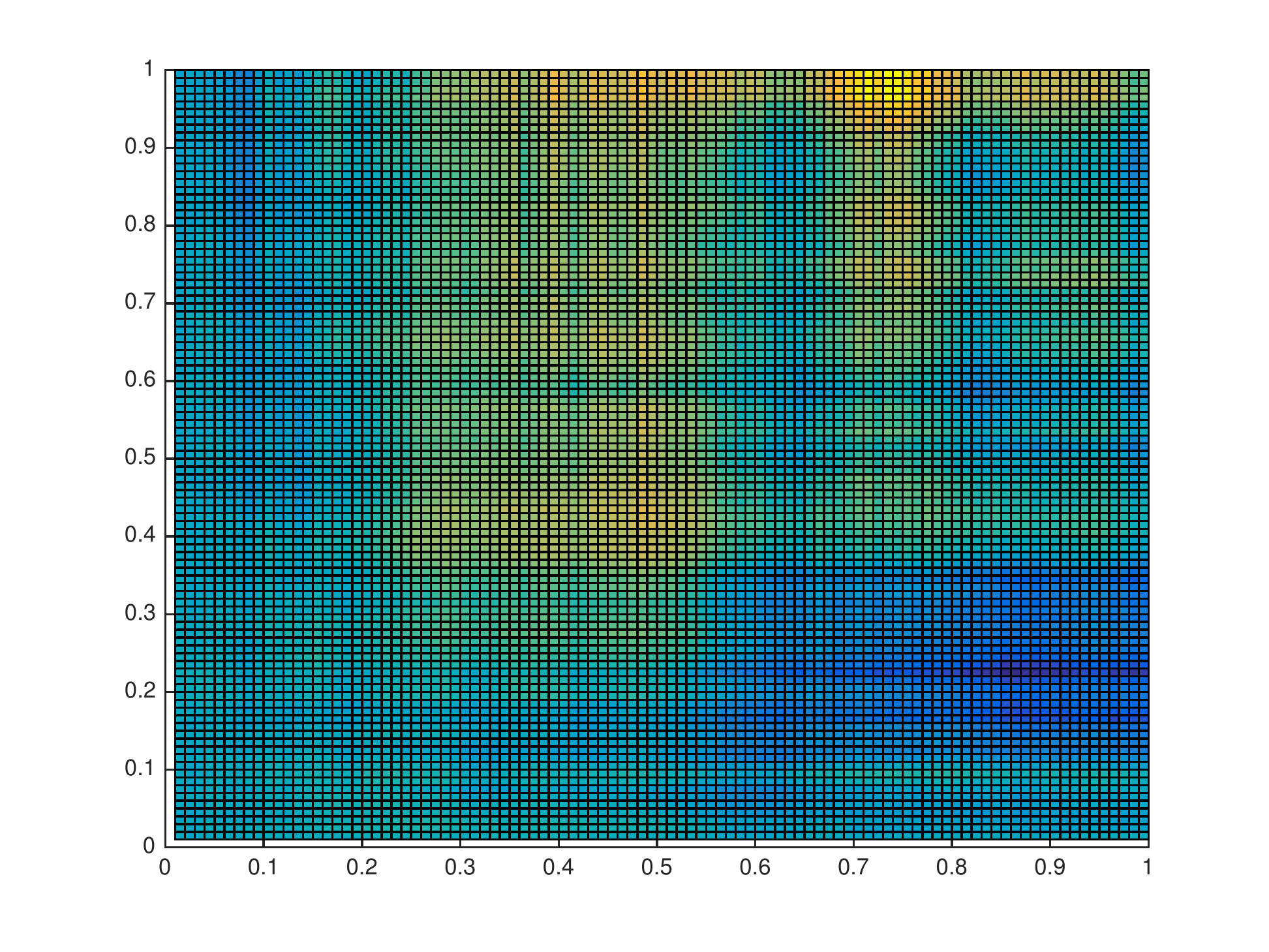}
\end{subfigure}
\caption{The Brown-Resnick field on $[0,1]^{2}$ with Brownian sheet input.
The grid mesh is 0.001.}
\label{fig:BRS}
\end{figure}

We validated the implementation of our algorithm by {\color{black}checking the distribution of $\max(M(0.5),M(1))$, which is the maximum of \thomasnote{the} Brownian-Resnick process at two locations with standard Brownian Motion generator, generated by our algorithm. Note that according to the bivariate 
\thomasnote{H\"usler-Reiss distribution} of this process, $\max(M(0.5),M(1))-\log (2\Phi(\sqrt{0.5}/2))$ should have the standard Gumbel distribution.} 
The QQ-plot in Figure~\ref{fig:qplot} confirms empirically that the distribution of $\max(M(0.5),M(1))-\log (2\Phi(\sqrt{0.5}/2))$ is indeed standard Gumbel.
\begin{figure}[tbp]
\centering
\includegraphics[width=5 in]{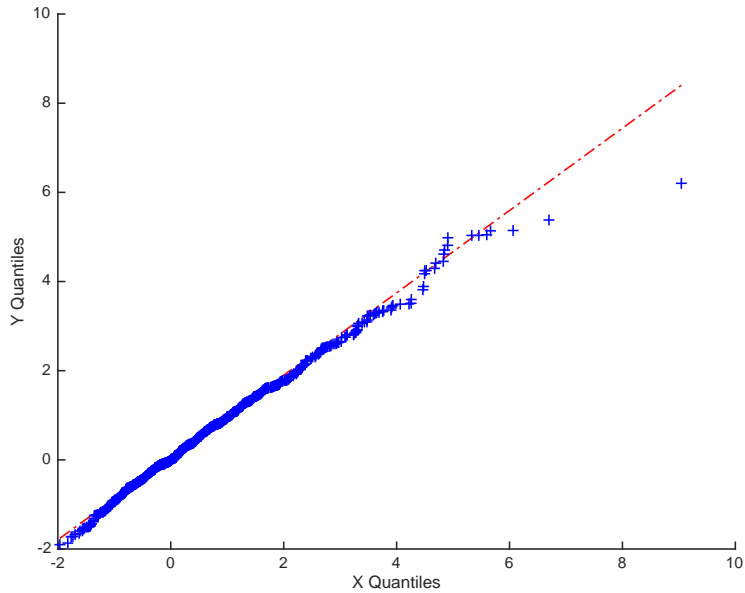}
\caption{The QQ-plot of $M(0.5)\bigvee M(1)-\log (2\Phi(\sqrt{0.5}/2))$ as generated by our algorithm vs. standard
Gumbel}
\label{fig:qplot}
\end{figure}

Next we will compare the computational cost in CPU time of our algorithm with the
algorithm proposed in \cite{DM}. We conducted both algorithms to generate
200 samples of the Brown-Resnick process \thomasnote{$M$} with fractional
Brownian motion inputs. We recorded both the average {\color{black}CPU} time for
generating a single sample and the $95\%$ confidence interval {\color{black}for the mean} based on our
200 samples, for different \thomasnote{grid numbers $d=1000,2000,5000$ and $10000$,} 
and with different Hurst \thomasnote{parameters $H\in \left\{
1/4,1/2,3/4\right\} $}
of the fractional Brownian motion. 
The sample estimates and the $95\%$ confidence intervals for the mean {\color{black}CPU}\ignore{\color{red}running} times to generate a single sample are shown in Table~\ref{tab:cost}. They illustrate that when the number of grids increases, the computational cost of our algorithm appears to increase
almost linearly, while the cost for the algorithm proposed in \cite{DM}
increases quadratically. Because we are using the circulant embedding method
to generate the fractional Brownian vectors, {\color{black} which has a complexity of order 
$O(d \log d)$}, \thomasnote{it is consistent with expectations.} \ignore{what we
expect.} \ignore{One thing}\thomasnote{It is worth noting} \ignore{to notice is that with such method,}\thomasnote{that for this method} the computational cost
to generate a \thomasnote{$d$}-dimensional Gaussian vector is the same as 
\thomasnote{for generating} a $2^{\lceil\log_2 d \rceil}$-dimensional Gaussian vector. However, this
consideration \thomasnote{will not} affect our comparison because we used this method in both
algorithms.

\begin{table}
\centering
\begin{tabular}{p{1.5cm}p{2.83cm}p{2.83cm}p{2.83cm}}
\hline
\multicolumn{4}{c}{Average cost per sample (second) (RB)}\\
\multicolumn{4}{c}{\zhipengnote{($\pm$ half-width of confidence interval)}}\\
\hline
$d$	&	$H=1/4$	&	$H=1/2$	&	$H=3/4$	\\\hline
1000	&	0.03 	$\pm$ 	0.003 	&	0.03 	$\pm$ 	0.002 	&	0.03 	$\pm$ 	0.001 	\\\hline
2000	&	0.08 	$\pm$ 	0.020 	&	0.06 	$\pm$ 	0.007 	&	0.06 	$\pm$ 	0.002 	\\\hline
5000	&	0.19 	$\pm$ 	0.071 	&	0.13 	$\pm$ 	0.004 	&	0.13 	$\pm$ 	0.008 	\\\hline
10000	&	0.32 	$\pm$ 	0.027 	&	0.26 	$\pm$ 	0.009 	&	0.27 	$\pm$ 	0.008 	\\\hline\end{tabular}

\begin{tabular}{p{1.5cm}p{2.83cm}p{2.83cm}p{2.83cm}}
\hline
\multicolumn{4}{c}{Average cost per sample (second) (DM)}\\
\hline
$d$	&	$H=1/4$	&	$H=1/2$	&	$H=3/4$	\\\hline
1000	&	0.40 	$\pm$ 	0.04 	&	0.28 	$\pm$ 	0.03 	&	0.43 	$\pm$ 	0.05 	\\\hline
2000	&	1.23 	$\pm$ 	0.13 	&	1.00 	$\pm$ 	0.13 	&	1.37 	$\pm$ 	0.15 	\\\hline
5000	&	7.32 	$\pm$ 	0.88 	&	4.82 	$\pm$ 	0.67 	&	5.97 	$\pm$ 	0.79 	\\\hline
10000	&	28.98 	$\pm$ 	3.18 	&	21.42 	$\pm$ 	2.64 	&	19.14 	$\pm$ 	2.67 	\\\hline
\end{tabular}

\begin{tabular}{p{1.5cm}p{2.83cm}p{2.83cm}p{2.83cm}}
\hline
\multicolumn{4}{c}{Average cost per sample (second) (EF)}\\
\hline
$d$	&	$H=1/4$	&	$H=1/2$	&	$H=3/4$	\\\hline
1000	&	0.15 	$\pm$	0.02 	&	0.13 	$\pm$	0.02 	&	0.15 	$\pm$	0.02 	\\\hline
2000	&	0.49 	$\pm$	0.06 	&	0.46 	$\pm$	0.05 	&	0.66 	$\pm$	0.09 	\\\hline
5000	&	2.83 	$\pm$	0.32 	&	2.34 	$\pm$	0.28 	&	3.39 	$\pm$	0.43 	\\\hline
10000	&	10.81 	$\pm$	1.46 	&	9.67 	$\pm$	1.24 	&	12.17 	$\pm$	1.70 	\\\hline
\end{tabular}

\caption{Comparison of running time of our algorithm (RB) vs. \protect\cite%
{DM} (DM) \zhipengnote{vs. \protect\cite{DEO} (EF)}.}
\label{tab:cost}
\end{table}

%
%
%
Next we compare the number of Gaussian vectors generated in our algorithm with the algorithms of \cite{DM} and \cite{DEO}. We \ignore{here} generate samples of the Brown-Resnick process with fractional 
Brownian \thomasnote{motion} generator, with $H=3/4$. 
We used the grid numbers $d=1000, 3000, 5000, 7000, 9000$. \zhipengnote{To get comparable relative error, we simulate 1000 times for algorithms DM and EF, and 10000 times for RB.} We calculated the sample average of the number of Gaussian vectors generated in each of the algorithms, and the $95\%$ confidence bounds. Table~\ref{tab:num} illustrates our main result. \zhipengnote{\thomasnote{On the left,} Figure~\ref{fig:num} \thomasnote{exhibits the plot corresponding to} Table~\ref{tab:num} for all three algorithms. \thomasnote{On the right, Figure~\ref{fig:num}  focuses on the 
algorithm RB.}} The number of Gaussian vectors generated increases linearly in both the algorithms of \cite{DM} and \cite{DEO}, with a reduction of constant factor using the \thomasnote{extremal} function algorithm from \cite{DEO}. In our algorithm \thomasnote{this number stays roughly at} the same level.
\begin{table}
\centering
%
%

\begin{tabular}{p{1 cm}p{3.3 cm}p{3.3 cm}p{3.3 cm}}
\hline
\multicolumn{4}{c}{Number of Gaussian vectors}\\
\hline
d	&	RB					&	DM						&	EF						\\\hline
1000	&	29.5 	$\pm$	2.0 	&	1522.1 	$\pm$	83.3 	&	1040.4 	$\pm$	60.5 	\\\hline
3000	&	28.7 	$\pm$	2.1 	&	4440.1 	$\pm$	248.9 	&	3101.1 	$\pm$	194.2 	\\\hline
5000	&	32.5 	$\pm$	4.2 	&	7648.0 	$\pm$	436.3 	&	5056.3 	$\pm$	298.2 	\\\hline
7000	&	31.4 	$\pm$	2.9 	&	10642.0 	$\pm$	638.4 	&	6961.4 	$\pm$	423.8 	\\\hline
9000	&	26.5 	$\pm$	1.5 	&	13570.0 	$\pm$	796.1 	&	8886.6 	$\pm$	510.3 	\\\hline
\end{tabular}
\caption{Comparison of number of Gaussian vectors generated in our algorithm
(RB) v.s. \protect\cite{DM} (DM) v.s. \protect\cite{DEO} (EF), $H=3/4$.}
\label{tab:num}
\end{table}

\begin{figure}[tbp]
\centering
\begin{subfigure}{.5\textwidth}
  \centering
  \includegraphics[width=3.5in]{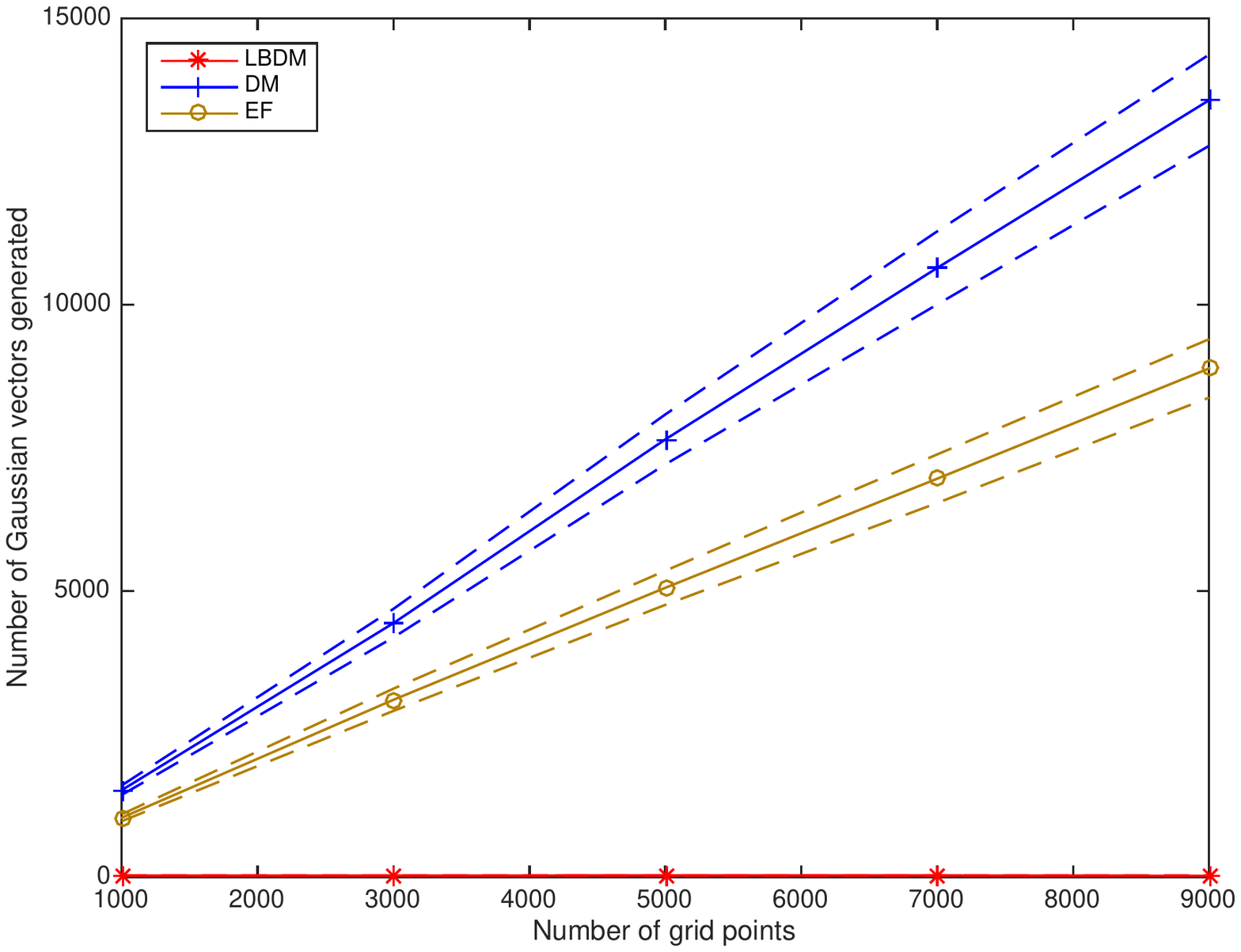}
\end{subfigure}%
\begin{subfigure}{.5\textwidth}
  \centering
  \includegraphics[width=3.5in]{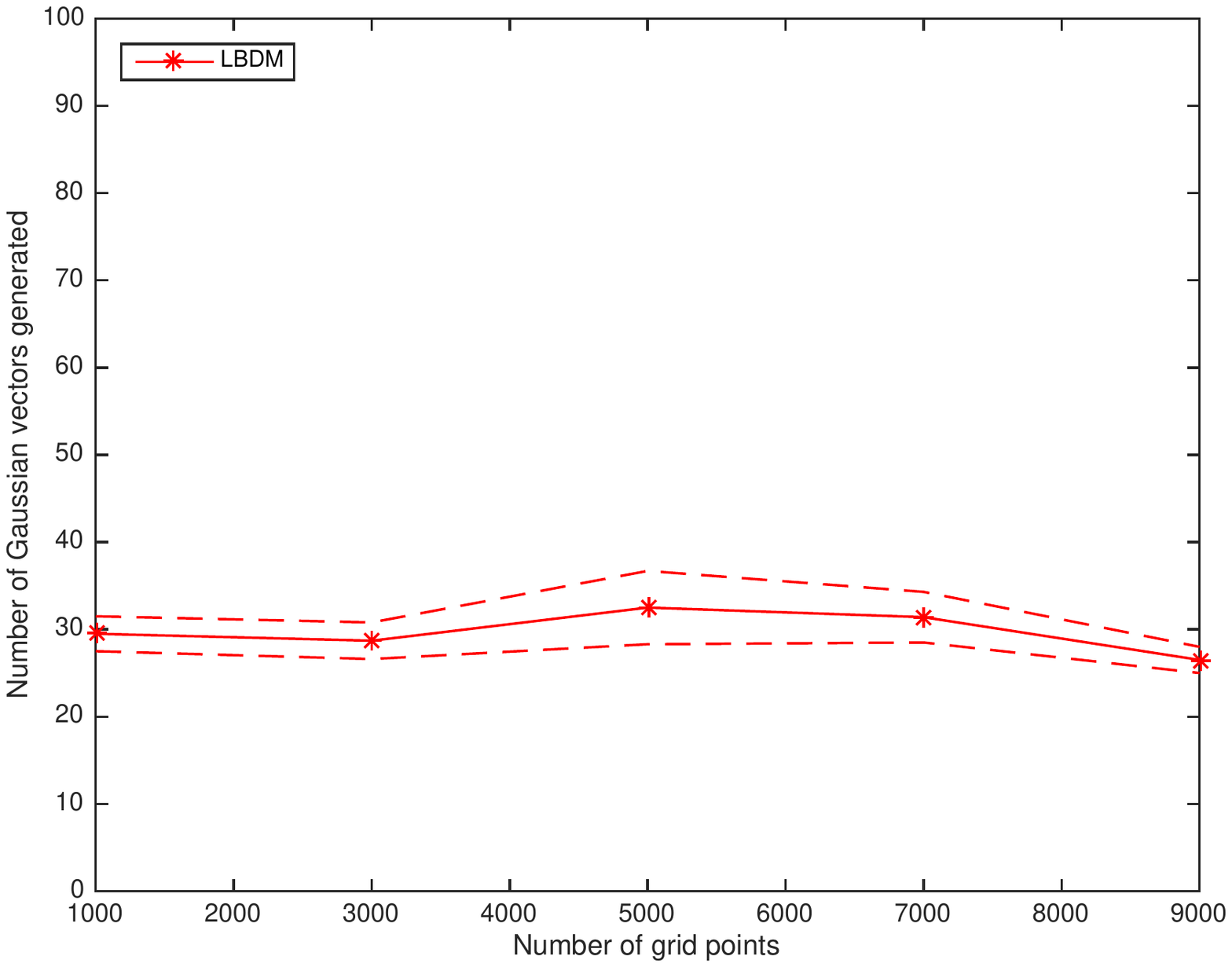}
\end{subfigure}\\
\caption{Comparison of number of Gaussian vectors generated in our algorithm
(RB) v.s. \protect\cite{DM} (DM) v.s. \protect\cite{DEO} (EF), $H=3/4$.}
\label{fig:num}
\end{figure}

\end{document}